\documentclass{article}
\makeatletter
\newcommand{\@affiliation}{}
\makeatother

\usepackage{arxiv}

\usepackage{lipsum}
\usepackage{amsfonts}
\usepackage{amsmath,amssymb}
\usepackage{graphicx}
\usepackage{epstopdf}
\usepackage{algorithmic}

\usepackage[utf8]{inputenc} 
\usepackage[T1]{fontenc}    
\usepackage{hyperref}       
\usepackage{url}            
\usepackage{color}

\usepackage{amsmath,amsfonts}
\usepackage{mathtools}
\numberwithin{equation}{section}
\usepackage{subcaption}
\usepackage{cleveref}  

\usepackage{xfrac}
\usepackage{csvsimple}
\usepackage{amsthm}
\usepackage{enumitem}
\newtheorem{theorem}{Theorem}[section]

\theoremstyle{definition}
\newtheorem{thm}{Theorem}[section]
\newtheorem{corollary}[thm]{Corollary}
\newtheorem{lemma}[thm]{Lemma}

\newtheorem{remark}[thm]{Remark}

\newcommand{\sgn}{\operatorname{sgn}} 
\newcommand{\ow}{\bar{w}}
\newcommand{\ds}{\displaystyle}

\newcommand{\rf}[1]{(\ref{#1})}
\newcommand{\eps}{\varepsilon}

\newcommand{\norm}[2]{{\| #1 \|}_{#2}}
\newcommand{\setstyle}[1]{{\mathbb #1}}
\newcommand{\dericonstant}{{\mathcal C}} 
\newcommand{\ue}{u^\varepsilon}
\newcommand{\we}{u^\varepsilon}
\newcommand{\oue}{{\bar{u}}^\varepsilon}
\newcommand{\psir}{{\bar{\psi}}}

\newcommand{\oqe}{{\bar{q}}^\varepsilon}
\newcommand{\ou}{\bar{u}}
\newcommand{\kernel}{{\mathcal K}}

\newcommand{\Newave}{N_{p,q}^{\eps}}
\newcommand{\Newaveo}{N_{p_0,q_0}^{\eps}}
\newcommand{\tildeNewave}{{\tilde N}_{p,q}^{\eps}}

\def\setR {\setstyle{R}}

\def\intd#1{{\,{\rm d}{#1}}}
\def\dx{\intd{x}}
\def\dxi{\intd{\xi}}

\def\dt{\intd{t}}

\def\ue{u^\eps}

\newcommand{\Ue}{U^\eps}
\newcommand{\oUe}{\bar U^\eps}

\title{ The Metastability Phenomenon for a \\Radiative Gas Model\thanks{C.R. thanks for financial support by the German Research Foundation (DFG), within  the project No. 526024901 in the Priority Programme - SPP 2410 Hyperbolic Balance Laws in Fluid Mechanics: Complexity, Scales, Randomness (CoScaRa).
}}

\newif\ifuniqueAffiliation
\uniqueAffiliationfalse

\ifuniqueAffiliation 
\author{ \href{https://orcid.org/0000-0000-0000-0000}{\includegraphics[scale=0.06]{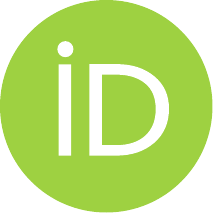}\hspace{1mm}David S.~Hippocampus}\thanks{Use footnote for providing further
		information about author (webpage, alternative
		address)---\emph{not} for acknowledging funding agencies.} \\
	Department of Computer Science\\
	Cranberry-Lemon University\\
	Pittsburgh, PA 15213 \\
	\texttt{hippo@cs.cranberry-lemon.edu} \\
	\And
	\href{https://orcid.org/0000-0000-0000-0000}{\includegraphics[scale=0.06]{orcid.pdf}\hspace{1mm}Elias D.~Striatum} \\
	Department of Electrical Engineering\\
	Mount-Sheikh University\\
	Santa Narimana, Levand \\
	\texttt{stariate@ee.mount-sheikh.edu} \\
}
\else
\usepackage{authblk}

\newbox{\orcid}\sbox{\orcid}{\includegraphics[scale=0.06]{orcid.pdf}} 

\author[]{%
Christian Rohde\thanks{Universit\"at Stuttgart,
Institute for Applied Analysis and
Numerical Simulation,
Pfaffenwaldring 57,
D-70569 Stuttgart, Germany, {\texttt{Christian.Rohde@mathematik.uni-stuttgart.de}}%
 } $\,$
and
	Adrian Viorel\thanks{Babe\c s-Bolyai University Cluj-Napoca, Departement of Mathematics, Str. Kogalniceanu 3, 400084 Cluj-Napoca, Romania, \texttt{adrian.viorel@ubbcluj.ro}}%
}
\fi

\renewcommand{\headeright}{The metastability phenomen for a radiative gas model}
\renewcommand{\undertitle}{}
\renewcommand{\shorttitle}{}

\hypersetup{
pdftitle={The metastability phenomen for a radiative gas model},
pdfsubject={math.AP},
pdfauthor={Christian Rohde, Adrian Viorel},
pdfkeywords={Radiative gas model, hhyperbolic balance laws, metastability, long-time behaviour},
}
\begin{document}
\maketitle
\begin{abstract}We study the time-asymptotic behaviour of solutions to a singular perturbation of the inviscid Burgers problem describing sound propagation in a radiating gas. This radiating gas model exhibits the
interesting phenomenon known as metastability that has been first identified  for the  viscous Burgers equation in \cite{KimTzavaras}. Precisely, on an initial time span, that scales linearly in the inverse of the perturbation parameter, we can quantify the $L^1$-distance   of the solution to the so-called
sharp $N$-wave solution of the inviscid Burgers problem and to  a diffusive $N$-wave. On a  second long-time stage that scales exponentially in the inverse of the perturbation parameter, the solutions still remain  close to a diffusive $N$-wave before finally collapsing into diffusion waves. Besides our analysis of the radiating gas model, we also provide an alternative description of metastability in the viscous Burgers model which, unlike other studies, relies on $L^1$-estimates and does not require the passage to similarity variables.
\end{abstract}
%
%
\section{Introduction}
In \cite{Hamer71} a simplified model for the nonlinear  propagation of sound  waves in a   radiating gas has been introduced.
The initial value problem  for this problem with unknowns gas state variable       $\oue=\oue(x,t)\in \setR$  and
tangential  radiative flux $\oqe=\oqe(x,t)\in \setR$  takes  in the so-called hyperbolic scaling (see \cite{LattanzioMarcati}) the form
\begin{equation}\label{RV-basicmodel}
\begin{array}{cl}
 \oue_t + {f(\oue)}_x = \eps \oqe_x,\quad
   -  \eps^2 \oqe_{xx} + \oqe = \oue_x  & \mbox{ in } \, \setR\times (0,\infty),\\[2ex]
\oue(\cdot,0) = u_0 & \mbox{ in } \, \setR.
\end{array}
\end{equation}
In \eqref{RV-basicmodel}, $\eps \in (0,1]$ is a fixed  absorption parameter    and $u_0:\setR\to \setR$ some initial function.
In particular we are interested in the nonlinear    flux choice
\[
 f(u) = \frac{u^2}{2},
\]
such that  the  mixed hyperbolic-elliptic  system   \rf{RV-basicmodel}  can be understood as a generalized
nonlocal Burgers equation for $\oue$.   A striking property (see \cite{Kawashima, LiuTadmor,RohdeXie}) is the conditional classical solvability of
\eqref{RV-basicmodel}. If $u_0$  and its  derivative  (supposed to exist)   satisfy certain  smallness conditions,  the problem
\eqref{RV-basicmodel} has a unique classical solution $(\oue,\oqe)$. For other initial data  the derivative of the solution
might  develop singularities and blow up  (see however  \cite{LattanzioMarcati} for the global existence of weak solutions).\\
Letting  $\eps\to 0$  in \rf{RV-basicmodel}, i.e., if we consider an optically thin limit regime,   we  are led to the  initial value problem for
a hyperbolic
conservation law, that is
\begin{equation}\label{RV-hyp}
\begin{array}{cl}
  u_t + {f(u)}_x =0  & \mbox{ in } \, \setR\times (0,\infty),\\[2ex]
u(\cdot,0) = u_0 & \mbox{ in } \, \setR.
\end{array}
\end{equation}
It is well-known that the  initial value problem \eqref{RV-hyp} requires a weak solution concept
regardless of the regularity of the initial datum.    Uniqueness of weak solutions can be enforced by imposing
additional entropy conditions.   In \cite{LattanzioMarcati} it is shown that $\oue \to u$ holds   in the space $L^1(\setR)$  such that the limit
function is an entropy solution in the sense of Kruzkov. It satisfies for sufficiently regular initial datum in particular
the one-sided Lipschitz (Oleinik) condition
\[
\frac{u(x+h) - u(x)}{h} \le \frac{1}{t} \quad (h>0)
\]
for all $t\in (0,\infty)$ and almost all $x\in \setR$.
If the numbers
\begin{equation}\label{RV-defpq}
p_0:= - \inf_{x\in \setR} \int_{-\infty}^x u_0(\xi) \dxi, \quad q_0 :=   \sup_{x\in \setR} \int^{\infty}_x u_0(\xi)\dxi
\end{equation}
exist, a special class of entropy solutions  for  $\eqref{RV-hyp}_1$   is given
by the  sharp $N$-waves
\begin{equation}\label{RV-nwave}
N_{p_0,q_0}(x,t) = \left\{
\begin{array}{r@{\quad:\quad}l}
x/t &  -\sqrt{2p_0}t < x < \sqrt{2  q_0} t,\\[1.5ex]
0& \mbox{elsewhere}.
\end{array}
\right.
\end{equation}
Most notably sharp $N$-waves   govern the time asymptotics of    entropy solutions  of initial value problems for
 \eqref{RV-hyp} (\cite{Hopf}, see Theorem \ref{theoNwave} below). \\
Much more common as \eqref{RV-basicmodel} is the   viscous
regularization  of \eqref{RV-hyp}. That is the   Burgers problem
\begin{equation}\label{RV-para}
\begin{array}{cl}
 \we_t + {f(\we)}_x = \eps \we_{xx}& \mbox{ in }  \setR\times (0,\infty),\\[2ex]
\we(\cdot,0) = u_0 & \mbox{ in }  \setR.
\end{array}
\end{equation}
The large time behavior of  the unique bounded classical  solution $\we$ of \rf{RV-para}  for fixed    $\eps$
is  governed  by  a diffusion wave of mass  $m_0 = q_0-p_0=\int_\setR u_0 \dx $. 
However,   a closer look  at the evolution     in  \cite{KimTzavaras}  reveals
an interesting    metastability phenomenon for the dynamics of  \eqref{RV-para}. At a first   stage $\we$  evolves from the initial state
into the an approximation of the saw-tooth profile of the $N$-wave.  This hyperbolic  structure persists for
a very long time (scaling with  the inverse of $\eps$), but on a later time scale the solution gets close to a diffusive $N$-wave. Thereafter, if the smallest of the positive and negative masses is
consumed, the solutions $\we$   transforms into the final asymptotic state of a diffusion wave.
This scenario has been verified rigorously  in  \cite{BeckWayne}, see also \cite{Wayne2}, where the distance between the solutions of \eqref{RV-para} and \eqref{RV-hyp}
is measured  pointwise in time with respect to a weighted $L^2$-norm. Nevertheless, a description of metastability based only on classical results and $L^1$-estimates is possible and, in Section \ref{sec:viscous} we propose such an approach as an alternative to \cite{BeckWayne}. This is one of the original contributions of our work, aimed not at improving \cite{BeckWayne}, to which our results are not directly comparable, but providing a simple framework in which metastability can be understood.  We note that further results 
on metastability for   conservation laws have been obtained in e.g.~\cite{Kim,MasciaStrani}.

The major issue of the present paper is then to extend  in Section \ref{sec:rad} the $L^1$-technique in order to analyze the radiative gas problem
\eqref{RV-basicmodel} in the case of
small  perturbation  parameter $\varepsilon$, i.e., solutions of
\eqref{RV-basicmodel} tend to  those of \eqref{RV-hyp} in the limit. After a review of relevant material from the literature in Section \ref{sec:prelimrad}
we consider a  time scale that is linear
in $\eps^{-1}$  in Section \ref{seclinear}.   Our first main results are Theorem   \ref{theo1} and Corollary \ref{cor1} that quantify the $L^1$-distances between the solution of \eqref{RV-basicmodel} and both,  
the sharp $N$-wave solution \eqref{RV-nwave} of  \eqref{RV-hyp} and a diffusive $N$-wave related to theviscous Burgers problem \eqref{RV-para}.
For longer time scales (up to times that scale exponentially in $\eps^{-1}$)   the novel results can be found in
Theorem  \ref{theo2} and Corollary \ref{cor2}
in Section \ref{secexp}. They show that the solution behaviour of \eqref{RV-basicmodel} is then
governed by a diffusive $N$-wave.
Unlike to the work in \cite{BeckWayne} our results are formulated in the $L^1$-norm with
respect to space.
\bigskip

We close the introduction with a review on previous  results on the time-asymptotic behaviour
of the radiative gas system \eqref{RV-basicmodel} and related systems.   $L^2$-decay rates for solutions of the radiative gas problem
have been proven in \cite{Laurencot}. The $L^1$-case has been analyzed in \cite{FrancescoLattanzio}.
These estimates concern the final asymptotic behaviour but not the metastable evolution
and have been generalized   to  decay   rates with respect to Sobolev space norms   later in \cite{Soloviev}.
Also the final asymptotic behaviour is the content of the following  papers   where
the stability of special solutions is discussed: \cite{Serre} for constant states and shock wave profiles in $L^1$
and \cite{JiZhang22,Lattanzioetal,Ohnawa} for   the stability  of shock waves profiles in $L^2$-/$L^\infty$-spaces. The long time behaviour of a Burgers equation with nonlocal flux is analysed in \cite{Coclite}.\\
None of these papers addresses the metastability effect on intermediate time regimes.
\bigskip

\textbf{Notations.}
We denote for some open set $D$ and $k\in \mathbb{N} \cup \{ \infty \}$ by  $C^k(D)$ the space of $k$-times continuously differentiable functions on $D$. The set $C^\infty_0(E)$ contains all smooth functions having compact support in $E\subset \bar D$. The set $C^k_b(E) \subset C^k(E)$ contains bounded functions.\\
Further, for $p \in [1,\infty) \cup \{\infty\}$, we denote the Lebesgue space on $D$ as $L^p(D)$ and denote the $L^p(D)$-norm as ${\lVert \cdot \rVert}_{L^p(D)}$, which for a function $f \in L^p(D)$ is defined by
\begin{align*}
    {\lVert f\rVert}_{L^p(D)} := \biggl(\int_{D} |f(x)|^p \,\dx\biggr)^{\frac{1}{p}} ~ \text{for}~p \in[1,\infty),
\end{align*}
and 
\[
{\lVert f\rVert}_{L^\infty(D)}
:=
\operatorname*{ess\,sup}_{x\in D}|f(x)|.
\]
Furthermore we need the Sobolov space $ W^{1,p}(\setR)$  consisting of  functions  such that the function itself and  the weak derivative is in $L^p(\setR)$. The set $BV(\setR)$ contains all functions $f\in L^1(\setR)$ with bounded BV-seminorm
\[
{|f|}_{BV(\setR)} := \sup_{h\neq0} \frac1h\int_\setR |f(x+h) -f(x)|    \dx.
\]
We denote by $C, \dericonstant, {\mathfrak C}, {\mathsf C}, M$  positive constants which depend 
on the initial datum $u_0$ but are independent of $\eps$. Their exact value might change from line to line. We use also the symbol "$\lesssim$" to indicate boundedness independent of $\eps$.

\section{The inviscid and the viscous Burgers problem\label{secprelim}}
In this section we discuss  the temporal asymptotics of the inviscid and the viscous Burgers problems \eqref{RV-hyp}  and \eqref{RV-para}.
Necessary results from the literature are reviewed  and we present our $L^1$-framework for the metastabilty phenomenen in the  viscous case.
\subsection{The inviscid Burgers problem}
We start to consider the large time behavior of 
entropy solutions for \eqref{RV-hyp}. Let $u_0 \in L^1(\setR) \cap L^\infty(\setR)$.\\
We call a  function $u\in L^\infty(\setR\times (0,\infty))$  an entropy solution if it  is a distributional solution of \eqref{RV-hyp} and satisfies 
      \[
       \int_\setR \int_0^\infty \eta(u) \phi_t + q(u)\phi_x \, \dt \dx    \ge
          - \int_\setR \eta(u_0) \phi(\cdot,0)  \dx
      \]
       for all nonnegative $\phi \in C^\infty_0(\setR\times [0,\infty))$ and all entropy/entropy-flux pairs $(\eta, q)$.\\
      By an entropy/entropy-flux pair $(\eta,q)$ we mean a smooth function $\eta:\setR \to \setR$ with $\eta''>0$ and a smooth function $q:\setR\to \setR$ satisfying $\eta ' u = q'$.\\
The entropy concept provides a uniqueness theory for distributional solutions of hyperbolic conservation laws. Moreover, entropy solutions have a dedicated  long time asymptotics. 
A seminal result shows that entropy solutions of the inviscid Burgers problem
always approach a sharp $N$-wave structure that depends only on the initial datum, precisely we have
\begin{theorem}\label{theoNwave}
 Let $u_0\in L^1 (\setR) \cap L^\infty(\setR)$
 have compact support.\\
The entropy solution $u$ of the initial value problem  \eqref{RV-hyp} satisfies for all $t>0$
\begin{equation}\label{inviscid-decay}
   \norm{(N_{p_0,q_0}-u)(\cdot,t)}{L^1(\setR)}   \lesssim    t^{-1/2}  ,
\end{equation}
 where $N_{p_0,q_0}$ has been defined in \eqref{RV-nwave}.
\end{theorem}
\begin{proof}
A proof can be found in \cite[Section 11.6]{Dafermos}.
\end{proof}
With other words, the $N$-wave governs the long-time behaviour  of the inviscid
Burgers problem \eqref{RV-hyp}.

\subsection{The viscous Burgers problem\label{sec:viscous}}
For the initial value problem  \eqref{RV-para} we summarize relevant  results on well-posedness
and time-asymptotic behaviour in  Theorems \ref{theopara}    and  Lemma \ref{lemma_decay}. Furthermore we present Kuznetsov-type estimates on the difference between 
the solutions of the viscous and the inviscid Burgers problem  in  Theorem \ref{theocontraction}. These results then pave the way for a short discussion of metastable phenomena for the viscous Burgers problem.
\begin{theorem}[Properties of $\ue$]\label{theopara}
Let $u_0\in L^1(\setR) \cap L^\infty(\setR) \cap BV(\setR)$.
\begin{itemize}[labelsep=0.5em, leftmargin=1cm,itemsep=0em, parsep=0em]
\item[(i)]
For all $\eps\in(0,1]$   there is a unique classical  solution $\we \in C^2_b(\setR\times (0,\infty))$ that
 satisfies for $t\ge 0$ the uniform bounds
\begin{equation}\label{para-lp}
\begin{array}{c}
\norm{\we(\cdot,t)}{L^p(\setR)} \le  \norm{u_0}{L^p(\setR)} \quad \big(p \in [1,\infty) \cup \{\infty\}\big),\\[1.5ex] 
{\vert \we(.,t)  \vert}_{BV(\setR)}  \le {\vert u_0 \vert}_{BV(\setR)}.
\end{array}
 \end{equation}
\item[(ii)]For $u_0 \in W^{1,\infty}(\setR)$ there is a   constant
$\dericonstant= \dericonstant\big(\norm{u_0}{L^1(\setR)},\norm{u_{0,x}}{L^\infty(\setR)}\big) 
$  independent of $\eps$ such that
\begin{equation}\label{paraest}
\begin{array}{c}
\ds \underset{x\in \mathbb{R}}{\sup} \Big\{{\oue}_{x}\left( x,t\right) \Big\} \leq \frac1t \quad  (t \in (0,t)),\\[2.1ex]
 \ds    \underset{x\in \mathbb{R}}{\sup} \Big\{{\oue}_{x}\left( x,t\right) \Big\} \leq
\frac{\dericonstant}{1+t} \quad  (t\in [0,\infty))
\end{array}
\end{equation}
hold. 
\end{itemize}
\end{theorem}
\begin{proof} 
The proofs  
can be found in standard text books, e.g.~\cite{Dafermos}.
\end{proof}

With the results in Theorem \ref{theopara} one can derive (sharper) decay estimates 
for norms of $\ue$ and $\ue_x$.  They are summarized in Lemma \ref{lemma_decay} below which is proven in the appendix.
To formulate it we define  for some function 
$w:\setR\to \setR$ the weighted integral 
\begin{equation} \label{moment_def}
M_2[w]
= \int_{\mathbb{R}} x^2 |w(x)| \dx.\end{equation}
\begin{lemma}[Time asymptotics for \eqref{RV-para}]\label{lemma_decay}
		Let the assumptions of Theorem \ref{theopara}(ii)  be satisfied.\\
  Then, 
there is a constant $C = C\big(\norm{u_0}{L^1(\setR)}, \norm{u_0}{L^\infty(\setR)}\big)>0$ independent of $\eps $ such that 
\begin{equation}
\left\Vert \ue\left( \cdot ,t\right) \right\Vert _{L^{\infty
}\left( \mathbb{R}\right) }\leq \frac{C}{\left( 1+t\right) ^{1/2}}
\label{Linftydecay}
\end{equation}
holds for all $t\ge 0$.\\  
Assume  additionally that  the assumptions on $u_0$ from Theorem  \ref{theopara}(ii)    and $M_2[u_{0,x}] < \infty$ hold.
Then, there exists a constant $M > 0$, depending only on $ \dericonstant$ from \eqref{paraest}, $C$  from \eqref{Linftydecay}, and 
$M_2[u_{x,0}]$, such that 
		\begin{equation}\label{L1decay_deri}
			{\|\ue_x(\cdot, t)\|}_{L^1(\setR)} \leq \frac{M}{(1+t)^{1/2}}.
		\end{equation}
holds for all $t \geq 0$.
\end{lemma}
\begin{proof} See in Appendix \ref{appendix}.
\end{proof}

From \cite{Serre}, we quote the following theorem due to Kuznetsov
that relates    $\we$ to the entropy solution  $u$ of \eqref{RV-hyp}. 

\begin{theorem}\label{theocontraction}
 Let $u_0\in L^1(\setR) \cap L^\infty(\setR) \cap  BV(\setR)$,  and let ${\{\we\}}_{\eps \in (0,1]}$ be  the  family of classical bounded solutions  of \eqref{RV-para}.  \\
 Then, the entropy solution $u$ of the initial value problem   \eqref{RV-hyp}  satisfies for all $t>0$
\begin{equation}\label{Kuznetsov}
   \norm{(\we-u)(\cdot,t)}{L^1(\setR)}   \lesssim   \eps^{1/2} t^{1/4}.
\end{equation}
\end{theorem}
Notice that the  estimate   \eqref{Kuznetsov}  is not uniform in time, and interestingly enough, the large time behavior of \eqref{RV-para} is governed by a diffusion
wave of mass $m_0=\int_\setR u_0  \dx$ (see \cite{EscobedoZuazua} or \cite{KimTzavaras}), and hence is very different from
that of the inviscid model \eqref{RV-hyp} described in Theorem \ref{theoNwave}. However, Kim and Tzavaras \cite{KimTzavaras} have shown
that on intermediate, $\eps $-dependent time scales
interesting dynamical phenomena occur, which can be summarized as follows: 
\begin{itemize}
\item[(a)]
during an ${\mathcal O}\left( \eps ^{-1}\right) $-long initial transient, the
solution of \eqref{RV-para} evolves from an arbitrary $u_{0}$ to an approximate
$N$-wave profile, called diffusive $N$-wave; 
\item[(b)] thereafter, this structure persists for an
exponentially  long time (with respect to the inverse of $\eps $), the
smallest of the negative and positive masses being slowly consumed until the
asymptotic state of a diffusion wave is reached.
\end{itemize}
Quite remarkably, the metastable manifold of diffusive $N$-waves can be represented explicitly via the Cole-Hopf transformation. It
consists  for $p,q\in \setR$  and some $t_0 >0$ of a family of functions $\Newave$ given by
\begin{equation} \label{diffNw}
\Newave\left( x,t; t_0\right) =\frac{\frac{x}{t}+\frac{B}{A}\sqrt{\frac{%
\eps}{\pi t_{0}}}}{1+\sqrt{\frac{t}{t_{0}}}\frac{1}{A}e^{\frac{%
x^{2}}{4\eps t}}-\sqrt{\frac{t}{t_{0}}}\frac{B}{A}e^{\frac{x^{2}}{%
4\eps t}}\frac{1}{\sqrt{\pi }}\int_{-\infty}^{x/\sqrt{4\eps t}%
}e^{-y^{2}}\intd{y}}.
\end{equation}
The free parameters $A,B$ can be chosen 
 such that the    diffusive $N$-wave satisfies at $t=t_0$ the mass relations
\[
p = -\inf_{x\in \setR} \int_{-\infty}^x  \Newave\left( \xi,t_0; t_0\right) \, \intd{\xi}
 \quad \mbox{and}\quad
q = \sup_{x\in \setR} \int^{\infty}_x  \Newave\left( \xi,t_0; t_0\right) \, \intd{\xi}
\]
hold. The    diffusive $N$-wave $\Newave$ is a solution of \eqref{RV-para} with inital datum taken as  $u_0= \Newave(\cdot,t_0;t_0)$. \\
For fixed positive times, these diffusive $N$-waves with masses $p=p_0,\, q=q_0$ are smooth approximations (hence their name) of the
sharp $N$-waves $N_{p_0,q_0}$ to which they converge as $\eps \rightarrow 0$, and we refer to \cite{KimTzavaras} for the explicit computation.

The existing analysis of metastable behavior, both in \cite{BeckWayne}
and in \cite{KimTzavaras}, relies essentially on the invariance of the Burgers
model under translations and scaling transformations, which suggest the
change of variables%
\newcommand{\wep}{w^{\varepsilon}}
\begin{equation}\label{scaling}
\tau =\ln \left( 1+t\right) ,\qquad \zeta =\frac{x}{\sqrt{1+t}},\qquad
\wep\left( \zeta ,\tau \right) =\sqrt{1+t}\,\ue\left( x,t\right),
\end{equation}%
giving rise to a scaled version of \eqref{RV-para}
\begin{equation}\label{svB}
\begin{array}{rcll}
\ds \wep_{\tau }+\frac{1}{2}\left( (\wep)^{2}-\zeta \wep\right) _{\zeta }&=&\eps \wep_{\zeta
\zeta } & \text{in } \setR\times (\infty), \\[1.5ex]
\wep(\cdot,0) &=& u_0 & \text{in } \setR.\\
\end{array}
\end{equation}
The advantage of this formulation lies in the fact that with respect to the
algebraically weighted Hilbert spaces%
\[
L^{2}_m\left( \setR\right) =\left\{ w\in L^{2}\left( \setR\right) :\left\Vert
w\right\Vert _{L^{2}_m\left( \setR\right) }^{2}=\int_\setR \left( 1+\zeta ^{2}\right)
^{m}\left\vert w\left( \zeta \right) \right\vert ^{2}\intd{\zeta} <\infty \right\} ,
\]%
with $m>1/2$, the operator defined by the linear part of \eqref{svB} has a
spectral gap between the continuous part of its spectrum and its isolated
eigenvalues. This permits\ the construction of associated invariant
manifolds, which plays a central role in the rigorous formulation of (a) and
(b) based on the rescaled version of the diffusive $N$-wave expressed in terms of the change of variables \eqref{scaling}
\begin{equation}\label{w-diffNw}
\tildeNewave\left( \zeta ,\tau; \tau_0:= \ln (1+t_0) \right) =\sqrt{1+t}\, \Newave
\left( x,t; t_0\right).
\end{equation}
We recall the two main results of Beck and Wayne \cite{BeckWayne},
which give a geometric description of metastability.

\begin{theorem}[Beck and Wayne \protect\cite{BeckWayne}]
\label{ThmWayne1}For  $m>3/2$  let $w_0\in L^2_m(\setR)$ and define the rescaled sharp  $N$-wave
$$\tilde{N}_{p,q}\left( \zeta \right)=\left\{
\begin{array}{r@{\quad:\quad}l}
\zeta &  -\sqrt{p} < \zeta < \sqrt{q},\\[1.5ex]
0& \mbox{elsewhere},
\end{array}
\right.
$$
with $p$ and $q$ determined by the initial datum
$
u_{0}
$ as in \rf{RV-defpq}.\\
Then, for  any $\delta >0$,  
there exists a positive number
$\tau _{1} = {\mathcal O}\left( \left\vert \ln \eps
\right\vert \right)$,  such that for the solution $\wep$  of  the initial value problem  \eqref{svB} we have
\[
{\left\Vert \wep(\cdot,\tau_1) - \tilde{N}_{p,q}(\cdot)\right\Vert}_{L^{2}_m \left(
\setR \right) }\leq \delta.
\]
\end{theorem}

\begin{theorem}[Beck and Wayne \protect\cite{BeckWayne}]\label{ThmWayne2}
For  $m>5/2$  let $\phi_0\in L^2_m(\setR)$ and define for some $\tau_0>0$
the rescaled diffusive   $N$-wave
\[
w_{0}^{N}(\zeta)= \tildeNewave\left( \zeta,\tau_0;\tau_0\right).
\]
Consider the problem \rf{svB} for initial datum $u_0\in L^2_m(\setR)$ with
\[
 w_0=w_{0}^{N}+\phi _{0}.
\]
If  $\norm{\phi_0}{L^2_m(\setR)}$ is small enough there is a constant $C_\eps>0$, such that
the solution $\wep$ of \rf{svB} satisfies
\[
\wep\left(\zeta, \tau \right) =\tildeNewave  \left( \zeta ,\tau; \tau_0 \right)+\phi \left(\zeta, \tau \right),
\]
and we have
\begin{equation} \label{Cbw}
\left\Vert \phi \left(\cdot,  \tau \right) \right\Vert _{L^{2}_m\left( \setR\right) }\leq
C_\eps e^{-\tau }
\end{equation}
with
$\tau >0$.
\end{theorem}
Difficulties arise however, due to the weighted norms, when one tries to express the above estimates in terms of the original variables $x,t$ and $\ue(x,t)$. Furthermore,
the constant  $C_\eps$  in \rf{Cbw} depends critically on $\eps$
(exponentially large in the inverse of $\eps$)  and can be very large for some initial conditions (cf.~Remark 5.2 in \cite{BeckWayne}).

{For the nonlocal problem \eqref{RV-basicmodel}   it is not clear whether an analogous $L^2_m(\setR)$-analysis is possible. In addition it could be preferable to
work with the (non-weighted) $L^1$-norm which is the basis for the entropy concept for scalar conservation laws.

To prepare our analysis on the radiative model problem we provide first
alternatives to Theorems \ref%
{ThmWayne1} and \ref{ThmWayne2}, which do not require the change of variables
\eqref{scaling} and  are formulated in the $%
L^{1}$-norm. In this framework  metastability for  the viscous Burgers model can be shown to rely only on the  well-known Kuznetsov-type
estimate and the estimate governing the asymptotic regime of entropy solutions to the inviscid Burgers model. We start with the analogon to Theorem \ref{ThmWayne1}.
}
\begin{theorem}
\label{Thm1}Let $u_{0}\in L^{1}\left( \mathbb{R}\right) \cap L^{\infty}\left( \mathbb{R}\right) \cap BV\left(
\mathbb{R}\right) $ have compact support.\\
Then, for the  classical solution $\ue$ of the initial value
problem \eqref{RV-para}, 
there exists a time $t_{0}={\mathcal O}\left( \eps
^{-1}\right) $, such that%
\[
\left\Vert (\ue -N_{p_{0},q_{0}})(\cdot, t_{0})
\right\Vert _{L^{1}\left( \mathbb{R}\right) }\lesssim\eps ^{1/4},
\]%
with $N_{p_{0},q_{0}}$ being the sharp $N$-wave defined in \eqref{RV-nwave}.
\end{theorem}
\begin{proof} Let $u$ be the entropy solution of the inital value problem \eqref{RV-hyp}.
From \eqref{inviscid-decay} and  \eqref{Kuznetsov} we have%
\begin{eqnarray*}
\left\Vert (\ue -N_{p_{0},q_{0}})\left(\cdot, t\right) \right\Vert
_{L^{1}\left( \mathbb{R}\right) } &\leq &\left\Vert (\ue -%
u)\left(\cdot, t\right) \right\Vert _{L^{1}\left( \mathbb{R}\right)
}+\left\Vert (u -N_{p_{0},q_{0}})\left(\cdot, t\right)
\right\Vert _{L^{1}\left( \mathbb{R}\right) } \\
&\lesssim &\eps ^{1/2}t^{1/4}+t^{-1/2},
\end{eqnarray*}%
and inserting $t=t_{0}= \mathcal{O}\left( \eps ^{-1}\right) $ above completes
the proof.
\end{proof}

The above theorem also gives an indication regarding the $L^1$-distance between the solution of \eqref{RV-para} at time $t_0= \mathcal{O}(\eps^{-1})$ and the metastable manifold
of  diffusive $N$-waves. Qualitatively, one can argue that $\left\Vert  \Newaveo  \left( \cdot,t_0; t_0\right) - N_{p_0,q_0} (\cdot,t_0) \right\Vert
_{L^{1}\left( \mathbb{R}\right) }$ can be made arbitrarily small for small $\eps$ based on the pointwise convergence (see \cite{KimTzavaras})
$$
\Newaveo\left( x,t_0; t_0\right) \rightarrow N_{p_0,q_0} (x,t_0), \quad \mbox{for} \quad \eps \rightarrow 0,
$$
and the dominated convergence theorem.

However, as both $\Newaveo$ and $N_{p_0,q_0}$ are known explicitly, a quantitative estimate of their distance can also be given.

\begin{remark}\label{dist}
For $t_{0}= \mathcal{O}(\eps ^{-1})$ the distance
between corresponding diffusive and sharp $N$-waves is
$$
\left\Vert    \Newaveo\left( \cdot,t_0; t_0\right) - N_{p_0,q_0} (\cdot,t_0) \right\Vert
_{L^{1}\left( \mathbb{R}\right) } \lesssim \eps,
$$
the two functions being much closer together than are $N_{p_0,q_0}$ and the solution $\ue$ of \eqref{RV-para}.\
To prove this claim we take a closer look at the formula \eqref{diffNw} defining diffusive $N$-waves, with
$t=t_0=1/\eps$,
for computational simplicity. We use the correct choice $A=e^{p_0/2\eps},\, B = 1-e^{p_0-q_0/(2\eps)}$. In this setting, \eqref{diffNw} rewrites as
\begin{equation}\label{diffNw-modif}
{\Newaveo}
\left( x,1/\eps; 1/\eps\right) = \frac{\eps x+\frac{B}{A}\frac{\eps}{\sqrt{\pi }}}{1+e^{\frac{1}{4}(x^{2}-2p_0/\eps)}(1-\Phi(x))+e^{\frac{1}{4}(x^{2}-2q_0/\eps)}\Phi(x)},
\end{equation}
where $\frac{B}{A}= \mathcal{O}(e^{-2p_0/\eps}+e^{-2q_0/\eps})$ is exponentially small and $\Phi$ is given by  $$\Phi(x)=\frac{1}{\sqrt{\pi }} \int_{-\infty}^{x/2} e^{-y^{2}}\intd{y} < 1.$$
In order to compare $\Newaveo(x,t_0=\eps^{-1};t_0)$  with the sharp $N$-wave   $N_{p_0,q_0}(x,t_0=\eps^{-1})=\eps x {\rm{1}}_{\left[-\sqrt{2p_0/\eps},\sqrt{2q_0/\eps}\right]}$  we  observe that
\begin{equation}\label{Jest}
0\leq \frac{1}{1+e^{\frac{1}{4}(x^{2}-2p_0/\eps)}(1-\Phi(x))+e^{\frac{1}{4}(x^{2}-2q_0/\eps)}\Phi(x)} \leq 1,
\end{equation}
and that the smooth function $x\mapsto \frac{1}{1+e^{\frac{1}{4}(x^{2}-2p_0/\eps)}(1-\Phi(x))+e^{\frac{1}{4}(x^{2}-2q_0/\eps)}\Phi(x)}$
is an approximation of the indicator function ${\rm{1}}_{\left[-\sqrt{2p_0/\eps},\sqrt{2q_0/\eps}\right]}$,
while its integral over $\mathbb{R}$ is of order $\mathcal{O}(\eps^{-1/2})$.\\
Bearing these  facts in mind, we can write
$$
\left\Vert \Newaveo\left( \cdot,t_0; t_0\right) - N_{p_0,q_0} (\cdot,t_0) \right\Vert_{L^{1}\left( \mathbb{R}\right)} = I_1 + I_2 + I_3,
$$
where
\begin{eqnarray*}
I_1 &:=& \int_{\left[-\sqrt{2p_0/\eps},\sqrt{2q_0/\eps}\right]} \frac{\eps |x| \left(e^{\frac{1}{4}(x^{2}-2p_0/\eps)}(1-\Phi(x))+e^{\frac{1}{4}(x^{2}-2q_0/\eps)}\Phi(x)\right)}%
{1+e^{\frac{1}{4}(x^{2}-2p_0/\eps)}(1-\Phi(x))+e^{\frac{1}{4}(x^{2}-2q_0/\eps)}\Phi(x)} \dx, \\
I_2 &:=& \int_{\mathbb{R}\setminus \left[-\sqrt{2p_0/\eps},\sqrt{2q_0/\eps}\right]} \frac{\eps |x|}%
{1+e^{\frac{1}{4}(x^{2}-2p_0/\eps)}(1-\Phi(x))+e^{\frac{1}{4}(x^{2}-2q_0/\eps)}\Phi(x)} \dx,\\
I_3 &:=& \int_{\mathbb{R}} \frac{\frac{B}{A} \frac{\eps}{\sqrt{\pi}}}%
{1+e^{\frac{1}{4}(x^{2}-2p_0/\eps)}(1-\Phi(x))+e^{\frac{1}{4}(x^{2}-2q_0/\eps)}\Phi(x)} \dx.
\end{eqnarray*}
The last of these three integrals scales as  $\mathcal{O}(e^{-2p_0/\eps}+e^{-2q_0/\eps})$, while the first and the second one can be estimated in terms of elementary integrals.

For $I_1$ we use \eqref{Jest} and integration by parts to conclude that
$$
I_1 \leq \int_{\left[-\sqrt{2p_0/\eps},\sqrt{2q_0/\eps}\right]} \eps |x| \left(e^{\frac{1}{4}(x^{2}-2p_0/\eps)}(1-\Phi(x))+e^{\frac{1}{4}(x^{2}-2q_0/\eps)}\Phi(x)\right) \dx = \mathcal{O}(\eps).
$$

On the other hand, for $I_2$ we use the convexity of $z\mapsto \frac{1}{1+z}$ and the fact that $\phi(x)\leq 1$ for any $x\in \setR$ to see that
$$
I_2 \leq \int_{\mathbb{R}\setminus \left[-\sqrt{2p_0/\eps},\sqrt{2q_0/\eps}\right]} \eps |x|\left( \frac{1-\Phi(x)}
{1+e^{\frac{1}{4}(x^{2}-2p_0/\eps)}}+\frac{\Phi(x)}
{1+e^{\frac{1}{4}(x^{2}-2q_0/\eps)}}  \right) \dx
$$
and reach the desired conclusion. 
\end{remark}

As diffusive $N$-waves change their positive and negative masses exponentially slow in time, the next result describes the second, slow motion, stage of the evolution.

\begin{theorem}\label{Thm2}
Let $u_0\in L^{1}\left( \mathbb{R}\right) \cap L^{\infty}\left( \mathbb{R}\right) \cap BV\left(
\mathbb{R}\right)$. \\
Consider for some $t_0>0$ and  $p_0,q_0$ from \eqref{RV-defpq}
the diffusive $N$-wave $\Newaveo
$ from \eqref{diffNw} and define
\[
u^N_0(x) =   \Newaveo\left( x,t_0; t_0\right).
\]
For some function $\psi^\eps_0 : \setR \to \setR$
let $\ue$ be the solution of the viscous
problem \eqref{RV-para} with initial data $
{u}_{0}=u_{0}^{N}+\psi^\eps_{0}$.\\
Then, $\ue$ remains at least as close as initially to the metastable manifold of diffusive $N$-waves for all future times in the sense that there is a function $\psi^\eps : \setR\times (0,\infty) \to \setR$ with 
\[
\ue \left(\cdot, t\right) =  \Newaveo\left( \cdot,t; t_0\right)  +\psi^\eps \left(\cdot, t\right)
\]%
and 
\begin{equation}
\left\Vert \psi^\eps \left(\cdot,  t\right) \right\Vert _{L^{1}\left( \mathbb{R}\right)
}\leq \left\Vert \psi^\eps_{0}\right\Vert _{L^{1}\left( \mathbb{R}\right) }
\label{L1contr}
\end{equation}
for all $t>0$.
\end{theorem}
In particular,  let us  take for the initial datum $u_0$ in Theorem \ref{Thm2}
the choice   ${u}_{0}(\cdot) = \ue(\cdot,t_0)$ to be the outcome of the first stage of the evolution for some $t_0 = \mathcal{O}(\eps^{-1})$ as in Theorem \ref{Thm1}.\\
Then, with $u_{0}^{N}(\cdot)=\Newaveo(\cdot,t_0;t_0)$, one has by the previous Theorem \ref{Thm1} and Remark \ref{dist} that
$$
\left\Vert \psi^\eps_0 \right\Vert _{L^{1}\left( \mathbb{R}\right)}\lesssim\eps^{1/4}
$$
and for all $t\geq t_0 = O(\eps^{-1})$
$$
\left\Vert \ue(\cdot,t) -   \Newaveo(\cdot,t,t_0) \right\Vert _{L^{1}\left( \mathbb{R}\right)}\lesssim \eps^{1/4}
$$
holds true. With other words we have verified the scenario for the second time stage.

{\emph{Proof of Theorem \ref{Thm2}.}}
Bearing in mind that both $\ue$ and $\Newaveo$ are solutions to the parabolic
conservation law \eqref{RV-para} corresponding to initial data $u_{0}$
respectively $u_{0}^{N}$, the $L^{1}$-contraction property of the nonlinear flow (see Theorem \ref{theopara} above or \cite[Theorem 6.3.2]{Dafermos}) %
readily implies that \eqref{L1contr} is true:%
\[
\left\Vert \psi^\eps \left(\cdot, t\right) \right\Vert _{L^{1}\left( \mathbb{R}\right)
}=\left\Vert \ue\left(\cdot, t\right) - \Newaveo\left(\cdot, t\right) \right\Vert
_{L^{1}\left( \mathbb{R}\right) }\leq \left\Vert u_{0}-u_{0}^{N}\right\Vert
_{L^{1}\left( \mathbb{R}\right) }=\left\Vert \psi^\eps_{0}\right\Vert
_{L^{1}\left( \mathbb{R}\right)}.
\]\hfill$\Box$

\section{The  radiative gas model\label{sec:rad}}
\subsection{Preliminaries on the  radiative gas model\label{sec:prelimrad}}
The system \eqref{RV-basicmodel}  for $(\oue,\oqe)$ can be rewritten as a  differential-integral
equation for $\oue$ alone.
The solution  $\oqe$ of the  linear elliptic equation in $\eqref{RV-basicmodel}_2$ can be represented using
\begin{equation}\label{K:def}
\kernel_\eps(x) = \frac{1}{\eps} K\Big( \frac{x}{\eps}  \Big),\quad
\kernel\left( x\right) =\frac{1}{2}e^{-\left\vert x\right\vert }.
\end{equation}
Then the   initial value problem   \eqref{RV-basicmodel} is equivalent to
\begin{equation}\label{RV-basicmodelequiv}
\oue_t + f(\oue)_x  = \frac{1}{\eps} ( \kernel_\eps\ast \oue - \oue), \quad \oqe = \kernel_{\eps,x}\ast \oue, \qquad  \oue(\cdot,0) = u_0.
\end{equation}
This fact is exploited in \cite{LattanzioMarcati, LiuTadmor}, and we collect the results
which are important for us. 
\begin{theorem}[Properties of $\oue$]\label{theorad}
Let $u_0\in L^1(\setR) \cap L^\infty(\setR) \cap BV(\setR)$.
\begin{itemize}
 \item[(i)] For all $\eps\in (0,1]$   there is a unique weak entropy solution
  $\oue\in L^\infty(\setR\times (0,\infty))$ of \eqref{RV-basicmodel} (respectively of  \eqref{RV-basicmodelequiv}), i.e., $\oue$ is a distributional solution of \eqref{RV-basicmodel} and satisfies 
      \[
      \begin{array}{rcl}
 \ds \!     \int_\setR \!\int_0^\infty \eta(\oue) \phi_t + q(\oue)\phi_x \, \dt \dx  &\!\!  \ge&\!\!
  \ds        - \frac{1}{\eps}  \int_\setR \!\int_0^\infty   (\kernel_\eps \ast \oue -\oue) \eta'(\oue) \phi   \,  \dt\dx \\[2.0ex]
   && {} \ds- \int_\setR \eta(u_0) \phi(\cdot,0)  \dx
    \end{array}
      \]
       for all $\phi \in C^\infty_0(\setR\times [0,\infty))$ and all entropy/entropy-flux pairs $(\eta, q)$.

  \item[(ii)] The family of unique  weak solutions ${\{\oue\}}_{\eps \in(0,1]}$  satisfies for $t\ge 0$ the uniform bounds
\begin{equation} \label{RV-LPuni}
\begin{array}{rcl}
\ds \norm{\oue(\cdot,t)}{L^p(\setR)} &\!\le &\! \ds \norm{u_0}{L^p(\setR)} \quad \big(p \in [0,\infty) \cup \{\infty\}\big),\\[1.5ex]
{\vert \oue(.,t)  \vert}_{BV(\setR)}  &\!\le&\! \ds  {\vert u_0 \vert}_{BV(\setR)}.
\end{array}
\end{equation}
 \item[(iii)]     If additionally $u_0\in C^1_b(\setR) \cap W^{1,\infty}(\setR) $ satisfies    for     $V( u_{0}) :=\sup_{x\in\setR} \big\{u_{0}(x)\big\}-\inf_{x\in\setR}\big\{ u_{0}(x)\big\}$
 the estimates
\begin{equation}\label{TLcond}
\begin{array}{rcl}
V\left( u_{0}\right) &\!<&\! \ds \frac{1}{4\kernel\left( 0\right) }\quad \mbox{and}\\[1.6ex]
\underset{x\in \mathbb{R}}{\inf } \Big\{ u_{0,x}\left( x\right)\Big\}
&\!>&\!\ds -\frac{1%
}{2\eps}\left( 1+\sqrt{1-4\kernel\left( 0\right) V\left(
u_{0}\right) }\right) 
,
\end{array}
\end{equation}
the weak solution $\oue$ satisfies   $\oue\in C^1_b(\setR\times [0,\infty))$, i.e., it is a classical solution.\\
There is a   constant
$\dericonstant= \dericonstant\big(\norm{u_0}{L^1(\setR)},\norm{u_{0,x}}{L^\infty(\setR)}\big)>0 
$  independent of $\eps$ such that
\begin{equation}\label{radest}
\begin{array}{rcl}
\ds \underset{x\in \mathbb{R}}{\sup} \Big\{{\oue}_{x}\left( x,t\right) \Big\} &\! \leq&\! \ds  \frac1t \quad  (t \in (0,\infty))
      \text{ and }\\[1.6ex]
      \ds\underset{x\in \mathbb{R}}{\sup} \Big\{{\oue}_{x}\left( x,t\right) \Big\} &\! \leq&\! \ds  \frac{\dericonstant}{1+t} \quad  (t\in [0,\infty))
\end{array}
\end{equation}
hold.

\end{itemize}
\end{theorem}

In the sequel we will mainly consider  initial data such that \eqref{RV-basicmodelequiv} has a classical solution, that means we suppose the assumptions of Theorem \ref{theorad}{\rm(iii)} to hold.
\smallskip

\parindent0em
{\textit{Proof}} of Theorem \ref{theorad}:
For statements (i), (ii)  we refer to  \cite{LattanzioMarcati}. The parts of statement (iii) concerning classical solvability and
the Oleinik-type estimate \eqref{radest}
can be found in \cite{LiuTadmor} and \cite[Theorem 5.1]{SchochetTadmor92}.  Note that the constant $\dericonstant$ is shown to depend on ($\norm{u_0}{L^1(\setR)}$ and) the positive parts of the  difference quotients $(u_0(x) -  u_0(y))/((x-y)$ only. 
\qed
\medskip

As in Section \ref{sec:viscous} for $\ue$ we provide a quantitative decay estimate for the $L^\infty$-norm and the $BV$-seminorm of $\oue(t)$ for  $t\ge 0$.

\begin{lemma}[Time asymptotics for \eqref{RV-basicmodel}]\label{lemma_decayradiative}
Let the assumptions of Theorem \ref{theorad}{\rm(iii)}  be satisfied.\\
Then, there is a constant $C = C\big(\norm{u_0}{L^1(\setR)}, \norm{u_0}{L^\infty(\setR)}\big)>0$ independent of $\eps $ such that 
\begin{equation}
\left\Vert \oue\left( \cdot ,t\right) \right\Vert _{L^{\infty
}\left( \mathbb{R}\right) }\leq \frac{C}{\left( 1+t\right) ^{1/2}}
\label{raddecay}
\end{equation}
holds for all $t\ge 0$.\\
Moreover, with the definition from \eqref{moment_def}, let $u_0$ satisfy 
\begin{equation}\label{inimoment}
M_2[u_{0,x}] < \infty.
\end{equation}
Then, there exists a constant $M > 0$, depending only on $C, \dericonstant$ and 
$M_2[u_{x,0}]$, such that
		\begin{equation}\label{L1decay_deri_rad}
			{\|\oue_x(\cdot, t)\|}_{L^1(\setR)} \leq \frac{M}{(1+t)^{1/2}}.
		\end{equation}
holds for all $t \geq 0$.
\end{lemma}
\begin{proof} See in Appendix \ref{appendix}.
\end{proof}
\subsection{Dynamics  of the radiative gas model at  \texorpdfstring{$t={\mathcal O}(\varepsilon^{-1})\!\!$}{t={\mathcal O}(eps^{-1})} \label{seclinear}}
In this section we analyze the  evolution of the \texorpdfstring{$L^1$}{L1}-distance between solutions of
\eqref{RV-para} and \eqref{RV-basicmodel} on a time scale that is linear in the inverse of the parameter $\eps$.\\
For the proof of our main result   we need the  integrated counterparts  %
\begin{equation}
U^\varepsilon_{t}+\frac{1}{2}\left( U^\varepsilon_{x}\right) ^{2}=U^\varepsilon_{xx}  \label{eq:int_visc}
\end{equation}%
of  the viscous Burgers model \eqref{RV-para} and the radiative gas model \eqref{RV-basicmodel}, that is
\begin{equation}
\begin{array}{rcccl}
\oUe_{t} & + & \ds \frac{1}{2}\left( \oUe_{x}\right) ^{2} & = &
\varepsilon \bar{Q}^\varepsilon_{x}, \\[1.5ex]
-\varepsilon ^{2}\bar{Q}^\varepsilon_{xx} & + & \bar{Q}^\varepsilon & = & \oUe_{x},%
\end{array}
\label{eq:int_rad}
\end{equation}%
where%
\begin{equation}
\begin{array}{c}
\ds 
U^\varepsilon\left( x,t\right) :=\int_{-\infty }^{x}\ue\left( \xi ,t\right) d\xi ,\,
\oUe\left( x,t\right) :=\int_{-\infty }^{x}\oue\left( \xi
,t\right) d\xi,\\[2.4ex]
\ds  \bar{Q}^\varepsilon\left( x,t\right) :=\int_{-\infty }^{x}\oqe\left( \xi
,t\right) d\xi.
\end{array}
\label{def:integrated}
\end{equation}%
We use the integrated (Hamilton-Jacobi-type)  equations in a similar way as Tadmor in \cite{Tadmor91}
and energy estimates as in \cite{EngelViorelRohde}   in order to derive the following  result on the $L^1$-distance between  solutions of the singularly perturbed problems  \eqref{RV-basicmodel} and \eqref{RV-para}.
\begin{theorem}[$L^1$-distance between $\ue$ and $\oue$]\label{theo1}
Let $u_0 \in L^1(\setR) \cap L^\infty(\setR) \cap BV(\setR)$ with $M_2[u_{0,x}] < \infty$ satisfy the assumptions in Theorem \ref{theopara}(ii) and in Theorem \ref{theorad}(iii).\\
Consider the classical solutions $u^\varepsilon$ and $(\oue,\oqe)$ of the  viscous Burgers problem \eqref{RV-para} and the radiative gas model \eqref{RV-basicmodel},  respectively. Then, there is a constant ${\mathfrak C}> 0$ that depends on the initial datum $u_0$ but not on $\eps$, such that the  
        $L^1$-distance between the solutions satisfies 
		\begin{equation}
			{\|(u^\varepsilon - \oue)(\cdot,t)\|}_{L^1(\mathbb{R})}  \leq {\mathfrak C}\varepsilon^{1/2} {(1+t)}^{1/4}
		\end{equation}
for all $t\ge 0$.	
	\end{theorem}
	
	\begin{proof}
		Using the integrated quantities from \eqref{def:integrated} and the equations \eqref{eq:int_visc}, \eqref{eq:int_rad} the 
		 integrated $L^1$ error $E^\eps(x,t) := \Ue(x,t) -  \oUe(x,t)$ satisfies 
		\begin{equation}
			E^\eps_t + \frac{1}{2}(u^\varepsilon + \oue) E^\eps_x = \varepsilon E^\eps_{xx} + \varepsilon (\oUe_{xx} - \bar{Q}^\varepsilon_x). \label{eq:error_evol}
		\end{equation}
		We multiply \eqref{eq:error_evol} by $\mathrm{sgn}(E^\eps)$ and integrate over $\mathbb{R}$ to obtain
	
\begin{equation}\label{inbetween}
        \begin{array}{c}
			\ds \frac{\rm d}{{\rm d}t} {\|E^\eps\|}_{L^1(\mathbb{R})} + \int_{\mathbb{R}} \frac{1}{2}(\ue + \oue) \partial_x |E^\eps| \dx\\[2.1ex] 
             \ds \hspace*{3cm} = \varepsilon \int_{\mathbb{R}} \mathrm{sgn}(E^\eps) E^\eps_{xx} \dx + \varepsilon \int_{\mathbb{R}} \mathrm{sgn}(E^\eps) (\oue_x - \bar{Q}^\varepsilon_x) \dx.
             \end{array}
		\end{equation}
        For any scalar function $f \in C^{2}(\mathbb{R})$, we have a.e.\ the Kato  inequality
	    \[
	    	\mathrm{sgn}(f) f_{xx} \leq \partial_{xx}|f|, 
	    \]	    
		by which  $\varepsilon \int_{\mathbb{R}} \mathrm{sgn}(E^\eps) E^\eps_{xx} \dx \leq 0$.

        For the radiative source term, inverting the elliptic operator in $-\varepsilon^2 \bar{q}^\varepsilon_{xx} + \bar{q}^\varepsilon = \bar{u}^\varepsilon_x$ we have $\bar{q}^\varepsilon = {\mathcal K}_\varepsilon * \bar{u}^\varepsilon_x$, that is $\varepsilon {\|\oue_x - \bar{Q}^\varepsilon_x\|}_{L^1(\mathbb{R})} 
			= \varepsilon {\|\oue_x - {\mathcal K}_\varepsilon * \oue_x\|}_{L^1(\mathbb{R})}$. 
		Using the triangle inequality, Young's convolution inequality \[
        {\|{\mathcal K}_\varepsilon * \bar{u}^\varepsilon_x\|}_{L^1(\mathbb{R})} \leq {\|{\mathcal K}_\varepsilon\|}_{L^1(\mathbb{R})}{\|\bar{u}^\varepsilon_x\|}_{L^1(\mathbb{R})} = {\|\bar{u}^\varepsilon_x\|}_{L^1(\mathbb{R})},\]
        as well as the decay property \eqref{L1decay_deri_rad} we get
		\begin{equation}
			\varepsilon {\|\oue_x - \bar{Q}^\varepsilon_x\|}_{L^1(\mathbb{R})} 
			 \leq \varepsilon {\|\oue_x\|}_{L^1(\mathbb{R})} + \varepsilon {\|{\mathcal K}_\varepsilon * \oue_x\|}_{L^1(\mathbb{R})}\leq 2\varepsilon {\|\oue_x\|}_{L^1(\mathbb{R})}
			\leq \frac{2  M\varepsilon}{(1+t)^{1/2}}. \label{eq:source_bound}
		\end{equation}
        
        To deal with the nonlinear convective term we integrate by parts and consider the Oleinik estimates for both viscous Burgers and radiative gas solutions (denoting the maximum of the two constants again by $\dericonstant$)
		\[
			\int_{\mathbb{R}} \frac{1}{2}(\ue + \oue) \partial_x |E^\eps| \dx = -\frac{1}{2} \int_{\mathbb{R}} (\ue_x + \oue_x) |E^\eps| \dx \geq -\min\left\{ \frac{1}{t},\frac{\dericonstant}{1+t} \right\} {\|E^\eps\|}_{L^1(\mathbb{R})}. 
		\]
        
        Here, neither $\frac{1}{t}$ nor $\frac{\dericonstant}{1+t}$ is completely satisfactory, as one is singular at the start of the evolution while the second introduces a $\dericonstant$-dependent decay rate. However, straightforward algebra gives
        \[\min\left\{\frac{\mathcal{C}}{1+t}, \frac{1}{t}\right\} \leq \frac{1}{1+t} + \frac{2\mathcal{C}}{(1+t)^2} \qquad \text{for all } t \ge 0, \]
        and returning with all estimates in \eqref{inbetween}, we have
        \[
        \frac{\intd{}}{\dt} {\|E^\eps\|}_{L^1(\mathbb{R})} \leq \left( \frac{1}{1+t} + \frac{2\mathcal{C}}{(1+t)^2} \right) {\|E^\eps\|}_{L^1(\mathbb{R})} + \frac{2 M \varepsilon}{(1+t)^{1/2}}.
        \]
        After multiplication with $(1+t)^{-1}$ and grouping of the exact derivative we are left with the differential inequality
        \[
        \frac{\intd{}}{\dt} \left( \frac{{\|E^\eps\|}_{L^1(\mathbb{R})}}{1+t} \right) \leq \frac{2\mathcal{C}}{(1+t)^2}\left( \frac{{\|E^\eps\|}_{L^1(\mathbb{R})}}{1+t} \right)  + \frac{2 M \varepsilon}{(1+t)^{3/2}},
        \]
        to which Gr\"onwall's inequality applies with giving
        \[
			\frac{{\|E^\eps\|}_{L^1(\mathbb{R})}}{1+t} \leq 2Me^{2\dericonstant}\eps \int_0^t (1+s)^{-3/2} \intd{s} \leq 4Me^{2\dericonstant}\eps,
		\]
        where we have used the identical initial data assumption, i.e., $E^\eps(\cdot,t)=0$. So, growth of the integrated error can be at most linear
        \begin{equation}
			{\|E^\eps(\cdot,t)\|}_{L^1(\mathbb{R})} \leq  C \varepsilon (1+t). \label{eq:E_final}
		\end{equation}
		We now apply the Gagliardo-Nirenberg inequality to the  $L^1$-norm of our target quantity $e^\eps := E^\eps_x = \ue - \oue$. In view of \eqref{L1decay_deri}  and \eqref{L1decay_deri_rad},  we note that the $L^1$-norm of $e^\eps_x$ can be estimated (denoting again the maximum of the two scaling constants by $M$) by 
		$$
			{\|e^\eps_x(\cdot, t)\|}_{L^1(\mathbb{R})} \leq {\|\ue_x(\cdot, t)\|}_{L^1(\mathbb{R})} + {\|\oue_x(\cdot, t)\|}_{L^1(\mathbb{R})} \leq \frac{2M}{(1+t)^{1/2}}.
		$$
		Substituting this and \eqref{eq:E_final} into the interpolation inequality leads to 
		\begin{equation}
			{\|e^\eps(\cdot, t)\|}_{L^1(\mathbb{R})} \leq 2  {\|E^\eps(\cdot, t)\|}_{L^1(\mathbb{R})}^{1/2} {\|e^\eps_x(\cdot, t)\|}_{L^1(\mathbb{R})}^{1/2}  \leq 2   \left(C \varepsilon (1+t) \right)^{1/2} \left( \frac{2M}{(1+t)^{1/2}} \right)^{1/2}. 
		\end{equation}
		This concludes the proof, as for ${\mathfrak C} = {\mathfrak C}(C,M) > 0$ we have
		\[
			{\| \ue(\cdot, t) - \oue(\cdot, t)\|}_{L^1(\mathbb{R})} \leq \mathfrak{C} \varepsilon^{1/2} (1+t)^{1/4}. 
		\]
	\end{proof}

Theorem \ref{theo1}  highlights only the $L^1$-distance between solutions of the singularly perturbed problems  \eqref{RV-basicmodel}  and  \eqref{RV-para}.
We are  interested in the distance to the solutions of the unperturbed problem \eqref{RV-hyp} on a time scale $t = {\mathcal O}(\eps^{-1})$.
If we combine the  results  from Theorem \ref{theo1} and  Theorem \ref{Thm1}, 
we get for the
$L^1$-distance between the  sharp  $N$-wave $N_{p_0,q_0}$  (as prototype solution \eqref{RV-nwave} of \eqref{RV-hyp}) and the solution
$\oue$  of $\eqref{RV-basicmodel}$  the  estimate
\begin{equation}\label{theo1est}
\begin{array}{rcl}
\lefteqn{\norm{(N_{p_0,q_0}-  \oue )(\cdot,t)}{L^1(\setR) }}\\[1.5ex]
& \le &\norm{(N_{p_0,q_0} 
	 -\ue)(\cdot,t)}{L^1(\setR)}
            + \norm{(\ue-\oue)(\cdot,t)}{L^1(\setR)} \\[1.5ex]
            &\lesssim &    \eps^{1/4}+         \eps^{1/2}   (1+t)^{1/4}.
\end{array}
\end{equation}
This leads us directly to the desired result
\begin{corollary}[$L^1$-distance between $\ue$ and $\oue$ for $t= {\mathcal O}(\eps^{-1})$]\label{cor1}
Let $u_0 \in L^1(\setR) \cap L^\infty(\setR) \cap BV(\setR)$ with $M_2[u_{0,x}] < \infty$   satisfy the   assumptions in Theorem \ref{theopara}(ii) and in Theorem \ref{theorad}(iii). \\
For the  temporal range  $t={ \mathcal O}(\eps^{-1})$  the solution $\oue$  of \eqref{RV-basicmodel}
approaches the sharp $N$-wave \eqref{RV-nwave} with $\eps\to 0$, i.e.,
\[
\norm{(N_{p_0,q_0}-\oue)(\cdot,t)}{L^1(\setR)}
\lesssim \eps^{\frac{1}{4}}.
\]
\end{corollary}

\begin{remark}
In the estimation \rf{theo1est}   we used that the  solution $\we$ of the viscous  Burgers problem  satisfies
\[
\norm{(N_{p_0,q_0}-\we)(\cdot,t)}{L^1(\setR)}
\lesssim   t^{-1/2}    + t^{1/4}\varepsilon ^{1/2},
\]
such that the statement of Corollary  \ref{cor1} is  likewise true for $\we$. This relies in particular on the  Kuznetsov estimate in Theorem
\ref{theocontraction}. A (probably technically more involved)   Kuznetsov-type estimate might    hold also for $\oue$ such that
Corollary \ref{cor1} can also proven this
way.
\end{remark}
\subsection{Dynamics of the radiative gas model  after \texorpdfstring{$t={\mathcal{O}}\left( \varepsilon^{-1}\right) $}{t=O(epsilon^{-1})}\label{secexp}}
As we have seen in the previous sections, up to a time of order $\eps^{-1}$, the solutions of the radiating gas model and those of the viscous Burgers model remain close
together and both approach a sharp
$N$-wave profile. It is also known (cf. \cite{EscobedoZuazua} and \cite{Laurencot}) that, as $t\rightarrow \infty$, the two models share the same asymptotical state, a
diffusion wave of total mass $M_0=q_0-p_0=\int_\setR u_0 \dx$, such that the only remaining question concerns the intermediate time scale  
on which the transition
from an approximate $N$-wave to the diffusion wave takes place. Hence, we will  analyze  our Cauchy problems starting from a time $t_0> 0$ (instead of $0$), which will be of order
$\mathcal{O}(\eps^{-1})$, and rely on the results of the previous section stating that at such a $t_0$ the solution of the radiating gas model is close
to the diffusive $N$-wave $\Newaveo(\cdot,t;t_0)$ defined by \eqref{diffNw}, having precisely the same positive and negative masses as the sharp $N$-wave $N_{p_0,q_0}$.
The function $\Newaveo$ is a
classical solution of the problem \rf{RV-para} to initial datum   $u_0 = \Newaveo(\cdot,t_0;t_0)$ and one readily shows that there is, for all $\eps\in (0,1]$, a constant ${\Pi^\infty_1>0} $ that is independent of $\eps$ and such that

\begin{equation}\label{RV-Piconstant}
\eps {\left\Vert \partial_x \Newaveo\left(\cdot, t,t_0\right) \right\Vert}_{L^{\infty } (\setR)}\leq \Pi^\infty_1
\end{equation}
holds for all $t\ge t_0    = {\mathcal O}(\eps^{-1})$.\\
We also need the following lemma on higher order derivatives which is proven in the appendix. 
The results holds only on a certain large-time time interval. Precisely,  let $c_1 = e^{-p_0/(2\eps)}$, $c_2 = e^{-q_0/(2\eps)}$ and define 
\begin{equation}\label{Teps}
	T_\eps := \left \{
    \begin{array}{c@{\,\,:\,\,}l}
    {\pi \eps^2 t_0}/{(c_1-c_2)^2} \ \ge \ \pi \eps^2 t_0 \, e^{\min\{p_0,q_0\}/\eps} & p_0\neq q_0,\\[1.2ex]
     \infty & p_0=q_0.
     \end{array}\right.
\end{equation}
The time $T_\eps$ is exponentially long with respect to the inverse of $\eps$ and marks the end of the time span on which $\Newaveo$ retains its $N$-wave shape, cf.~(b) in Section \ref{sec:viscous}; for $t \gg T_\eps$ the diffusive $N$-wave has turned into a diffusion wave.
\begin{lemma}[Higher order $L^1$-estimates for the diffusive $N$-wave] \label{prop:Nwave_deriv}
For all integers $k\ge 1$ and some $t_0={\mathcal O}(\eps^{-1})$ there exist constants $\Pi^1_{k} > 0$ depending only on $k$, $p_0$ and $q_0$ but not on $\eps$, such that the spatial derivatives of the diffusive $N$-wave $\Newaveo $ from \eqref{diffNw}   satisfy for \begin{equation}\label{timespan}t\in {\color{black}I_\eps:= [t_0, T_\eps)}
\end{equation}
the estimate 
\begin{equation*}
	\norm{\partial_x^k \Newaveo(\cdot,t;t_0)}{L^1(\setR)} \le \Pi^1_{k}\, \eps^{1-k}\, t^{-\frac{k}{2}}.
\end{equation*}
		\end{lemma}

The next theorem shows that once the solution $\oue$ of the radiative model gets close to the metastable manifold of diffusive $N$-waves described by \eqref{diffNw}, it will remain close to this manifold for all times. In other words, the two dissipative models, the well-known viscous one \eqref{RV-para} and its nonlocal counterpart \eqref{RV-basicmodel} have solutions which remain remarkably close to each other even for large times.

\begin{theorem}[Evolution of $L^1$-distance between $\ue$ and $\oue$ after $t = {\mathcal O}(\eps^{-1})$]
\label{theo2} Let the initial datum $u_0 \in L^1(\setR) \cap L^\infty(\setR) \cap BV(\setR)$  satisfy the conditions
from  \eqref{TLcond}  in Theorem \ref{theorad}(iii).\\
For $p_0,q_0$ from \eqref{RV-defpq} let
the  solution $ \oue$ of \eqref{RV-basicmodelequiv} at time $t_0$ be split according to
\begin{equation*}
\oue(x,t_0)=  \Newaveo\left( x,t_0; t_0\right) +\psir^\eps_{0}(x),
\end{equation*}
with some function $\psir^\eps_0\in C^1(\setR)$ and the diffusive $N$-wave $\Newaveo$ from \eqref{diffNw}.\\
Then, at future times $t \in I_\eps$, there is a   function $\psir^\eps \in C^1([t_0,\infty)\times \setR)$  with $\psir^\eps(\cdot,t_0) = \psir^\eps_0$ and a constant ${\mathsf C}={\mathsf C}(\Pi^1_4)>0$
such that $\oue $ 
\eqref{RV-basicmodel} 
can be represented as%
\begin{equation*}
{\oue}\left(x,t \right) = \Newaveo\left( x,t; t_0\right)  +\psir^\eps \left(x,t\right) ,
\end{equation*}%
and  $\psir^\eps$   satisfies 
the estimate
\begin{equation*}
{\left\Vert \psir^\eps \left(\cdot, t\right) \right\Vert}_{L^{1} (\setR)}   - {\left\Vert \psir^\eps
_{0}\right\Vert}_{L^{1}(\setR)} \leq  \eps \mathsf  C.
\end{equation*}%
\end{theorem}

\begin{proof} For the sake of simplicity we write $N^\varepsilon = \Newaveo\left(\cdot,\cdot; t_0\right) $
for the diffusive $N$-wave, which for $t\ge t_0$ satisfies
\begin{align}
			N^\varepsilon_t + \partial_x {\bigg(\frac{1}{2}(N^\varepsilon)^2\bigg)} &= \varepsilon N^\varepsilon_{xx} \label{eq:N_wave}
\end{align}
by construction. Subtracting \eqref{eq:N_wave} from \eqref{RV-basicmodelequiv}, gives an equation for the error
\newcommand{\psire}{\psir^\eps}
$\psire = \oue - N^\varepsilon$
		\begin{equation}
			\psire_t + \partial_x \left( \frac{1}{2}(\oue)^2 - \frac{1}{2}(N^\varepsilon)^2 \right) = \frac{1}{\varepsilon}({\mathcal K}_\varepsilon * \oue - \oue) - \varepsilon N^\varepsilon_{xx}. \label{eq:psi_evol}
		\end{equation}
Introducing the remainder term        
\begin{equation}
			R^\varepsilon:= \frac{1}{\varepsilon}(\kernel_\varepsilon * N^\varepsilon - N^\varepsilon) - \varepsilon N^\varepsilon_{xx} \label{eq:remainder_def}
		\end{equation}
the equation rewrites with $\psire$ dissipation
\[
			\psire_t + \partial_x \left( \frac{1}{2}(\oue)^2 - \frac{1}{2}(N^\varepsilon)^2 \right) = \frac{1}{\varepsilon}(\kernel_\varepsilon * \psire - \psire)  + R^\eps. 
		\]
		We multiply the last equation  by $\mathrm{sgn}(\psire)$ and integrate over $\mathbb{R}$ to obtain
		\begin{equation}
        \begin{array}{c}
			\ds \frac{\rm d}{{\rm d}t} {\|\psire(t)\|}_{L^1(\mathbb{R})} + \int_{\mathbb{R}} \mathrm{sgn}(\psire) \partial_x \left( \frac{1}{2}(\oue)^2 - \frac{1}{2}(N^\varepsilon)^2 \right) \dx\\[2.9ex]
         \hspace*{2cm} 
         \ds  
            = \frac{1}{\varepsilon} \int_{\mathbb{R}} \mathrm{sgn}(\psire) (\kernel_\varepsilon * \psire - \psire) \dx + \int_{\mathbb{R}} \mathrm{sgn}(\psire) R^\varepsilon \dx. 
            \end{array}\label{eq:L1_full}
		\end{equation}
		   Here, one can rewrite the nonlinear term on the left hand side as a derivative. Indeed, with $\frac{1}{2}(\oue)^2 - \frac{1}{2}(N^\varepsilon)^2 = \frac{1}{2}\left(\oue + N^\varepsilon \right)\psire $, we have $\mathrm{sgn}(\psire)\partial_x\left(\frac{1}{2}\left(\oue + N^\varepsilon \right)\psire\right) = \partial_x \left( \frac{1}{2}\left(\oue + N^\varepsilon \right) |\psire| \right)$ a.e. Since $\oue(x,\cdot)$ and $N^\varepsilon(x,\cdot;t_0)$ decay to zero for $|x| \to \infty$, integrating the exact derivative yields
		\begin{equation}
			\int_{\mathbb{R}} \mathrm{sgn}(\psire) \partial_x \left( \frac{1}{2}(\oue)^2 - \frac{1}{2}(N^\varepsilon)^2 \right) \dx =  \int_{\mathbb{R}}\partial_x \left( \frac{1}{2}\left(\oue + N^\varepsilon \right) |\psire| \right)\dx
            = 0.
		\end{equation}
		Next, we consider the dissipative term on the right hand side of \eqref{eq:L1_full}.   By Fubini's theorem  we get  
		\begin{align*}
			\int_{\mathbb{R}} \mathrm{sgn}(\psire(x))  & \big( \kernel_\varepsilon * \psi(x) - \psire(x) \big) \dx\\[1.1ex]
            &= \int_{\mathbb{R}} \mathrm{sgn}(\psire(x)) \int_{\mathbb{R}} \kernel_\varepsilon(x-y) \psire(y) \, {\rm d}y \dx - {\|\psire\|}_{L^1(\mathbb{R})} \\[1.1ex]
			&\leq \int_{\mathbb{R}} \int_{\mathbb{R}} \kernel_\varepsilon(x-y) |\psire(y)| \dx \, {\rm d}y - {\|\psire\|}_{L^1(\mathbb{R})} \\[1.1ex]
			&= \int_{\mathbb{R}} |\psire(y)| \left( \int_{\mathbb{R}} \kernel_\varepsilon(x-y) \dx \right) {\rm d}y - {\|\psire\|}_{L^1(\mathbb{R})} \\[1.1ex]
			&= {\|\psire\|}_{L^1(\mathbb{R})} - {\|\psire\|}_{L^1(\mathbb{R})} = 0.
		\end{align*}
		In view of these computations, \eqref{eq:L1_full} reduces to the simple
		\begin{equation}
			\frac{\rm d}{{\rm d}t} {\|\psire(t)\|}_{L^1(\mathbb{R})} \leq {\|R^\varepsilon(\cdot, t)\|}_{L^1(\mathbb{R})}.\label{eq:diff_ineq_psi}
		\end{equation}
To deal with the reminder term, as in the proof of Theorem \ref{theo1}, we return to the inversion of the elliptic operator $-\varepsilon^2 \partial_{xx} + \mathrm{Id}$ by convolution  with $\kernel_\varepsilon$. We use the resolvent identity
		\[
			\kernel_\varepsilon * f - f = \varepsilon^2 \kernel_\varepsilon * f_{xx}, 
		\]
twice: first on the parenthesis in \eqref{eq:remainder_def}, then on the resulting difference
\[
R^\varepsilon = \frac{1}{\varepsilon}(\kernel_\varepsilon * N^\varepsilon - N^\varepsilon) - \varepsilon N^\varepsilon_{xx} = \varepsilon \kernel_\varepsilon * N^\varepsilon_{xx} - \varepsilon N^\varepsilon_{xx} = \varepsilon^3 \kernel_\varepsilon * N^\varepsilon_{xxxx}.
\]
Now using Young's convolution inequality (${\|\kernel_\varepsilon\|}_{L^1(\setR)} = 1$) and Lemma \ref{prop:Nwave_deriv} with $k=4$ results in
		\[
			{\|R^\varepsilon(\cdot, t)\|}_{L^1(\mathbb{R})} \leq \varepsilon^3 {\|N^\varepsilon_{xxxx}(\cdot, t)\|}_{L^1(\mathbb{R})}\lesssim \varepsilon^3 \frac{1}{\varepsilon^3 t^2}.
		\]
Returning to \eqref{eq:diff_ineq_psi} and integrating from the the metastability onset $t_0 = {\mathcal O}(\varepsilon^{-1})$ to  $t_0 < t \in I_\eps $ completes the proof 
		\begin{equation*}
			{\|\psire(t)\|}_{L^1(\mathbb{R})} - {\|\psire(t_0)\|}_{L^1(\mathbb{R})} \lesssim \int_{t_0}^t s^{-2} \intd{s} \lesssim \eps. 
		\end{equation*}
        
\end{proof}
Finally, as a consequence of the  major conclusion Corollary \ref{cor1} of Section \ref{seclinear}, which states that, on a ${\mathcal O}(\eps^{-1})$-time scale,
the solution $\oue$ is separated from the sharp $N$-wave \eqref{RV-nwave} only by an error of order $\eps^{1/4}$, we see that Theorem \ref{theo2} applies in the interesting case
$$t_0={ \mathcal O}(\eps^{-1})\quad \mbox{and} \quad \norm{\psir^\eps_0}{L^1(\setR)}\lesssim \eps^{1/4},$$
in which  
$\partial_x \Newaveo$ satisfies the required boundedness condition \eqref{RV-Piconstant} with $t\in I_\eps$.
In fact, Theorem \ref{theocontraction}, Remark \ref{dist} and Theorem \ref{theo1} imply that
\[
\begin{array}{rcl}
 \norm{\psir^\eps_0}{L^1(\setR)} &\!=& \!\norm{\Newaveo(\cdot,t_0;t_0)-\oue(\cdot,t)}{L^1(\setR)} \\[1.2ex]
 &\!\leq&\! \norm{\Newaveo(\cdot,t_0;t_0) - N_{p_0,q_0}(\cdot,t_0)}{L^1(\setR)} + \norm{N_{p_0,q_0}(\cdot,t_0) -\oue(\cdot,t_0)}{L^1(\setR)}\\[1.2ex]
 &\!\lesssim& \!\eps^{1/4},
\end{array}
\]
and as a direct consequence of Theorem \ref{theo2}, 
we have
\begin{corollary}[$L^1$-distance between the diffusive  $N$-wave
and $\oue$ after $t = {\mathcal O}(\eps^{-1})$]\label{cor2}
Let $u_0 \in L^1(\setR) \cap L^\infty(\setR) \cap BV(\setR)$   satisfy the conditions
from  \eqref{TLcond}  in Theorem \ref{theorad}(iii).\\
For the  temporal range  $t \in I_\eps$  the solution $\oue$  of \eqref{RV-basicmodel}
remains close to the metastable manifold of the viscous Burgers problem \eqref{RV-para}, i.e.,
\[
\norm{(\Newaveo-\oue)(\cdot,t)}{L^1(\setR)}
\lesssim 
\eps^{1/4} 
\]
holds. Here $\Newaveo$ is the diffusive $N$-wave given in \eqref{diffNw},  which
at time $t_0={ \mathcal O}(\eps^{-1})$ has as positive and negative masses $p_0,q_0$ precisely those of the sharp $N$-wave $N_{p_0,q_0}$.
\end{corollary}
\section{Conclusions}
We have shown that the  radiating gas model
exhibits a  metastability behaviour similar to that of the   viscous Burgers equation. On a time scale
that is linearly  inverse  with respect to  the perturbation parameter $\eps$, solutions we quantify the distance  to  $N$-wave solutions 
of the inviscid  and the viscous Burgers problem in terms of the perturbation parameter.
At a  later stage with end time scaling exponentially in the inverse of $\eps$ the solutions remain close to a diffusive $N$-wave. 
Our analysis works within an $L^1$-framework. 
Future work will be devoted to the extension of the metastability results to other $L^p$-spaces and the investigation of the 
optimality of the obtained bounds. It is an interesting 
question whether the Burgers scenario holds also for other dissipative mechanisms like 
relaxation approximations  or nonlocal diffusive-dispersive approximations \cite{JinXin,Rohde}.
\appendix
\section{Appendix}
\label{appendix}
We provide the proofs for Lemmas \ref{lemma_decay} and Lemma
\ref{lemma_decayradiative} on the decay of solutions to \eqref{RV-para} and \eqref{RV-basicmodel} with respect to the $L^1$-norm of spatial
derivatives. Furthermore we prove 
Lemma \ref{prop:Nwave_deriv} on 
estimates for the diffusive $N$-wave.

\textit{Proof} of Lemma \ref{lemma_decay}:
We skip the upper index $\eps$ in the classical solution $\ue$ of \eqref{RV-para}.\\
The $L^{\infty}$-decay estimate \eqref{Linftydecay}
is a direct consequence of the first
inequality in \eqref{paraest} (see also \cite{Dafermos}
and the  estimate \eqref{para-lp} for $p=\infty$.\\
The {proof} of \eqref{L1decay_deri}  is based on several steps.

\textbf{Step 1:} (Inner integral)\\
 The function  $w := u_x$  satisfies for $t>0$   the equation
	\begin{equation} \label{eq:diff_burgers}
		w_t + (uw)_x = \varepsilon w_{xx}.
	\end{equation}
	We split the  space derivative $w$ into its positive and negative parts $w = w^+ - w^-$. Note that $|w| = w^+ + w^- = 2w^+ - w$. For $R > 0$, we integrate $|w|$ over $[-R, R]$
	\begin{equation} \label{eq:inner_split}
		\int_{-R}^R |w(x,t)| \dx = 2 \int_{-R}^R w^+(x,t) \dx - \int_{-R}^R w(x,t) \dx,
	\end{equation}
	such that, by the fundamental theorem of calculus, the second term can be estimated 
    by 
	\begin{equation}\label{a1}
		-\int_{-R}^R w(x,t) \dx = u(-R, t) - u(R, t) \leq 2 {\|u(\cdot, t)\|}_{L^\infty(\setR)} \leq \frac{2C}{(1+t)^{1/2}}.
	\end{equation}
 Here we used     \eqref{Linftydecay}.\\
	For the first term in \eqref{a1}, we apply the Oleinik-like bound \eqref{paraest}. From  $w^+ \leq  \frac{\dericonstant}{1+t}$ we deduce
	\begin{equation}\label{a2}
		2 \int_{-R}^R w^+(x,t) \dx \leq 2 \int_{-R}^R \frac{\dericonstant}{1+t} \dx = \frac{4R\dericonstant}{1+t}.
	\end{equation}
	Combining  \eqref{a1}  and \eqref{a2} yields
	\begin{equation} \label{eq:inner_boundpara}
		\int_{-R}^R |w(x,t)| \dx \leq \frac{4R\dericonstant}{1+t} + \frac{2C}{(1+t)^{1/2}}.
	\end{equation}
	We choose now   $R$ to be the time-dependent function $R(t) := K(1+t)^{1/2}$  with  an explicit constant {\color{black}$K > 0$ to be chosen later}. Substituting $R(t)$ into \eqref{eq:inner_boundpara} gives
	\begin{equation} \label{eq:inner_dynamicpara}
		\int_{-R(t)}^{R(t)} |w(x,t)| \dx \leq \frac{4K\dericonstant + 2C}{(1+t)^{1/2}}.
	\end{equation}
	
	\textbf{Step 2:} (Tail integral and global bound)\\
	We abbreviate the time-dependent second-moment expression $M_2[w(\cdot,t)]$   from \eqref{moment_def} by 
$M_2(t) := M_2[w(\cdot,t)] 
	 $.
	Using Chebyshev's inequality 
$$ {\color{black}\int_{|x| > R} f(x) \dx \leq \frac{1}{R^2} \int_{|x| > R} x^2 f(x)  \dx},$$
 for  non-negative functions $f=f(x) \geq 0$ and any positive threshold $R > 0$,
we bound the tail integral as
	\begin{equation} \label{eq:outer_boundpara}
		\int_{|x| > R(t)} |w(x,t)| \dx \leq \frac{1}{R(t)^2} \int_{|x| > R(t)} x^2 |w(x,t)| \dx \leq \frac{1}{K^2(1+t)} M_2(t).
	\end{equation}
	Combining \eqref{eq:inner_dynamicpara} and \eqref{eq:outer_boundpara}, we obtain the global bound 
	\begin{equation} \label{eq:L1_boundpara}
		{\|w(\cdot,t)\|}_{L^1(\setR)} \leq \frac{4K\dericonstant + 2C}{(1+t)^{1/2}} + \frac{1}{K^2(1+t)} M_2(t).
	\end{equation}	
	\textbf{Step 3:} (Differential inequality for $M_2(t)$)\\
	We multiply the differentiated Burgers equation \eqref{eq:diff_burgers} by
    $x^2 \operatorname{sgn}(w(x,t))$ and integrate over $\mathbb{R}$
    to obtain
	\begin{equation}
		\frac{\rm d}{\dt} M_2(t) = -\int_{\mathbb{R}} x^2 \sgn(w) (uw)_x \dx + \varepsilon \int_{\mathbb{R}} x^2 \sgn(w) w_{xx} \dx =: I_C + I_D.
	\end{equation}
	 With the product rule we get $(uw)_x = u_x w + u w_x = w^2 + u w_x$ and compute for the integrand in $I_C$
	\begin{equation} x^2\sgn(w) (uw)_x =
		x^2  \sgn(w) w^2 + x^2 \sgn(w) u w_x  = x^2 w|w| + x^2 u \partial_x |w|.
	\end{equation}
	This leads to
	\begin{equation}\label{Ic}
		I_C = -\int_{\mathbb{R}} x^2 w|w| \dx - \int_{\mathbb{R}} x^2 u \partial_x |w| \dx.
	\end{equation}
	For the second term in \eqref{Ic} integration by parts yields
    \[
		-\int_{\mathbb{R}} x^2 u \partial_x |w| \dx = \int_{\mathbb{R}} \partial_x(x^2 u) |w| \dx 
     \]
     and therefore 
     \[
      -\int_{\mathbb{R}} x^2 u \partial_x |w| \dx   = \int_{\mathbb{R}} (2xu + x^2 w) |w| \dx = 2\int_{\mathbb{R}} xu|w| \,dx + \int_{\mathbb{R}} x^2 w|w| \dx.
     \]   
	Thus, for $I_C$, the $\int x^2 w|w| \dx$ term cancels out and we get
	\[
		I_C = 2\int_{\mathbb{R}} xu|w| \dx.
	\]
	For the dissipative term $I_D$, applying {\color{black}Kato's inequality ($\sgn(w)w_{xx} \le \partial_{xx}|w|$)} and integrating by parts twice yields
	\[
		I_D = \varepsilon \int_{\mathbb{R}} x^2 \sgn(w) w_{xx} \dx \le \varepsilon \int_{\mathbb{R}} x^2 \partial_{xx}|w| \dx = 2\varepsilon \int_{\mathbb{R}} |w| \dx = 2\varepsilon {\|w\|}_{L^1(\setR)}.
	\]
	Thus, combining these gives the differential inequality for the second moment
	\begin{equation}
		\frac{\rm d}{\dt}   {M}_2(t) \leq 2 {\|u\|}_{L^\infty(\setR)} \int_{\mathbb{R}} |x| |w| \dx + 2\varepsilon {\|w\|}_{L^1(\setR)}.
	\end{equation}
	By the Cauchy-Schwarz inequality and the definition of $M_2$ we obtain \[\int_\setR |x| |w| \dx \leq \|w\|_{L^1(\setR)}^{1/2} M_2(t)^{1/2},\] 
    which leads with the amplitude bound \eqref{Linftydecay} to
	\begin{equation} \label{eq:M2_diff_ineqpara}
      \frac{\rm d}{\dt}   {M}_2(t) \leq \frac{2C}{(1+t)^{1/2}} \|w(t)\|_{L^1}^{1/2} M_2(t)^{1/2} + 2\varepsilon {\|w(t)\|}_{L^1(\setR)}.
	\end{equation}
	
	\textbf{Step 4:} (Boundedness of an  auxiliary function)\\ 
	At this point, it is convenient to introduce
	\begin{equation}
		Z(t) := (1+t)^{-1/2} M_2(t) \implies M_2(t) = Z(t)(1+t)^{1/2},
	\end{equation}
	and perform the  subsequent computations in terms of this quantity.
	Differentiating gives \[
   \frac{\rm d}{\dt} {M}_2(t) = 
   \frac{\rm d}{\dt}
   {Z}(t)(1+t)^{1/2} + \frac{1}{2}Z(t)(1+t)^{-1/2}.\] Substituting this into \eqref{eq:M2_diff_ineqpara} and multiplying the entire equation by $(1+t)^{1/2}$ yields
	\begin{equation} \label{eq:Z_diffpara}
	\ds 	(1+t) \frac{\rm d}{\dt}   {Z}(t) + \frac{1}{2} Z(t) \leq 2C {\|w(t)\|}_{L^1(\setR)}^{1/2} Z(t)^{1/2} (1+t)^{1/4} + 2\varepsilon {\|w(t)\|}_{L^1(\setR)} (1+t)^{1/2}.
	\end{equation}
Let $\tilde C = 4K\dericonstant + 2C$. 
    From \eqref{eq:L1_boundpara} we deduce ${\|w(t)\|}_{L^1(\setR)} \leq \frac{\tilde C + K^{-2} {Z}(t)}{(1+t)^{1/2}}$ and by  subadditivity ${\|w(t)\|}_{L^1(\setR)}^{1/2} \leq \frac{\sqrt{\tilde C} + K^{-1} \sqrt{Z(t)}}{(1+t)^{1/4}}$. Using this relation  in \eqref{eq:Z_diffpara} leads us to
	\[
    \begin{array}{rcl}
	\ds 	(1+t)  \frac{{\rm d}}{{\rm d}t}   {Z}(t) &\leq& 
        \ds \left( \frac{2C}{K} +  \frac{2\varepsilon}{K^2} - \frac{1}{2} \right) {Z}(t) +    {2C} \sqrt{\tilde C} {Z}(t)^{1/2} + 2\varepsilon \tilde C\\[2.5ex]
        &\leq& \ds  \ds  \underbrace{\left( \frac{2C}{K} +  \frac{2}{K^2} - \frac{1}{2} \right)}_{=:\, -\gamma\, =\,- \gamma(K)} {Z}(t) +    {2C} \sqrt{\tilde C} { Z}(t)^{1/2} + 2\tilde C.
     \end{array}   
        \label{Zineq}
	\]
    We used $\eps\in (0,1]$.
	Now we choose the free parameter $K$ such that $\gamma >0$ holds. We get 
	\begin{equation}\label{Zineq2}
		(1+t) \frac{{\rm d}}{{\rm d}t}    {Z}(t) \leq -\gamma {Z}(t) + \beta {Z}(t)^{1/2} + \delta =: F(Z(t)),
	\end{equation}
	with $\beta := 2C\sqrt{\tilde C}$ and $\delta := 2 C_1$.
    The right hand side $F$ of \eqref{Zineq2} is strictly negative for all numbers  $ Z > Z_{max}$, where $Z_{max}$ 
	\[
		Z_{max} := \left( \frac{\beta + \sqrt{\beta^2 + 4\gamma\delta}}{2\gamma} \right)^2.
	\]
	By standard ODE comparison principles, ${Z}(t) \leq \max\{{M}_2[u_{0,x}], Z_{max}\} =: M_Z$ for all $t \geq 0$.\\
	Returning to ${\|u_x(t)\|}_{L^1(\setR)}$ and using $M_2(t) \leq M_Z (1+t)^{1/2}$, we obtain
	\begin{equation} \label{eq:final_L1_bound}
	{\|u_x(t)\|}_{L^1(\setR)} =	{\|w(t)\|}_{L^1(\setR)} \leq \frac{4K + 2C}{(1+t)^{1/2}} + \frac{M_Z}{K^2 (1+t)^{1/2}} = \frac{M}{(1+t)^{1/2}},
	\end{equation}
	where $M := 4K + 2C + \frac{M_Z}{K^2}$.
	This completes the proof.
\qed
\bigskip 

We proceed with the proof of Lemma \ref{lemma_decayradiative} on  $\oue$ and spatial derivatives of $\oue$.
It follows  the strategy that has been developped  in the proof of Lemma \ref{lemma_decay}.
   
  \textit{Proof} of Lemma \ref{lemma_decayradiative}:
We skip  the upper index $\eps$ in the classical solution $\oue$ of \eqref{RV-basicmodelequiv}.\\
The $L^{\infty}$-decay estimate \eqref{raddecay}
is a direct consequence of the first
inequality in \eqref{radest} (see also \cite[Proposition 3]{Laurencot}) and the  estimate \rf{RV-LPuni} for $p=\infty$.\\
We proceed with the verification of \eqref{L1decay_deri_rad}.\\
\textbf{Step 1:} (Inner integral bound)\\
Let $\ow = \ou_x$. Differentiating the equivalent equation  \eqref{RV-basicmodelequiv} with respect to $x$, the evolution is governed by
	\begin{equation} \label{eq:diff_rad}
		\ow_t + (\ou\ow)_x = \frac{1}{\varepsilon}({\mathcal K}_\varepsilon * \ow - \ow).
	\end{equation}
	We split $\ow$ into its positive and negative parts $\ow = \ow^+ - \ow^-$. Since $|\ow| = 2\ow^+ - \ow$, we integrate for any number $R > 0$ over $[-R, R]$ and obtain
	\begin{equation} \label{eq:inner_splitrad}
		\int_{-R}^R |\ow(x,t)| \dx = 2 \int_{-R}^R \ow^+(x,t) \dx - \int_{-R}^R \ow(x,t) \dx.
	\end{equation}
 For the second term we have with \eqref{raddecay}  the decay estimate 
	\[
		-\int_{-R}^R \ow(x,t) \dx = \ou(-R, t) - \ou(R, t) \leq 2 {\|\ou(\cdot, t)\|}_{L^\infty(\setR)} \leq \frac{2C}{(1+t)^{1/2}}.
	\]
	For the first term, applying the Oleinik-like bound \eqref{radest} yields
	\[
		2 \int_{-R}^R \ow^+(x,t) \dx \leq 2 \int_{-R}^R \frac{\mathcal C}{1+t} \,dx = \frac{4R\mathcal C}{1+t}.
	\]
	For the time-dependent number $R(t) := K(1+t)^{1/2}$ with some free parameter $K > 0$ we get altogether
	\begin{equation} \label{eq:inner_dynamic}
		\int_{-R(t)}^{R(t)} |\ow(x,t)| \dx \leq \frac{4K\dericonstant + 2C}{(1+t)^{1/2}}.
	\end{equation}
	
	\textbf{Step 2:} (The tail integral and global bound)\\
	We recall \eqref{moment_def} and define
	\begin{equation} \label{eq:M2_def}
		{\bar M}_2(t) :=   M_2[\bar w(\cdot, t)] =\int_{\mathbb{R}} x^2 |\ow(x,t)| \dx.
	\end{equation}
	Using Chebyshev's inequality, we bound the tail mass 
	\begin{equation} \label{eq:outer_bound}
		\int_{|x| > R(t)} |\ow(x,t)| \dx \leq \frac{1}{R(t)^2} \int_{\mathbb{R}} x^2 |\ow(x,t)| \dx = \frac{1}{K^2(1+t)} {\bar M}_2(t).
	\end{equation}
	Combining \eqref{eq:inner_dynamic} and \eqref{eq:outer_bound}, we obtain the global bound on ${\|\ow(t)\|}_{L^1(\setR)}$ given by
	\begin{equation} \label{eq:L1_bound}
		{\|\ow(t)\|}_{L^1(\setR)} \leq \frac{4K\dericonstant + 2C}{(1+t)^{1/2}} + \frac{1}{K^2}  \frac{{\bar M}_2(t)}{(1+t)}  .
	\end{equation}
	
	\textbf{Step 3:} (Differential inequality for ${\bar M}_2(t)$)\\
	We multiply the differentiated equation \eqref{eq:diff_rad} by $x^2 \sgn(\ow(x,t))$ and integrate over $\mathbb{R}$ to obtain
	\begin{equation}
		\frac{\rm d}{\rm dt} {\bar M}_2(t) = -\int_{\mathbb{R}} x^2 \sgn(\ow) {(\ou\ow)}_x \,dx + \frac{1}{\varepsilon} \int_{\mathbb{R}} x^2 \sgn(\ow) ({\mathcal K}_\varepsilon * \ow -\ow ) \dx  = :\bar I_C + \bar I_D.
	\end{equation}
	For the local convective term $I_C$ we get as in the proof of Lemma \ref{lemma_decay} 
    $$\bar I_C = 2\int_{\mathbb{R}} x \ou|\ow| \dx.$$
	For the non-local dissipative source term $\bar I_D$, we utilize a non-local analogue of Kato's inequality. Since $\sgn(\ow(x)) \ow(y) \leq |\ow(y)|$ for all $x,y$, and ${\mathcal K}_\varepsilon \ge 0$, we have
	\[
		\sgn(\ow(x)) ({\mathcal K}_\varepsilon * \ow(x) - \ow(x)) \leq {\mathcal K}_\varepsilon * |\ow|(x) - |\ow(x)|.
	\]
	Thus, applying Fubini's theorem to $\bar I_D$ yields
	\begin{equation}
		\bar I_D \leq \frac{1}{\varepsilon} \int_{\mathbb{R}} x^2 ({\mathcal K}_\varepsilon * |\ow|(x) - |\ow(x)|) \dx = \frac{1}{\varepsilon} \int_{\mathbb{R}} |\ow(y)| \left( \int_{\mathbb{R}} x^2 {\mathcal K}_\varepsilon(x-y) \dx - y^2 \right) {\rm d}y.
	\end{equation}
	We evaluate the inner integral by the substitution $z = x-y$
    \newcommand{\dz}{\,{\rm d}z}
	\begin{equation}
		\int_{\mathbb{R}} (y+z)^2 {\mathcal K}_\varepsilon(z) \dz = y^2 \int_{\mathbb{R}} {\mathcal K}_\varepsilon(z) \,dz + 2y \int_{\mathbb{R}} z {\mathcal K}_\varepsilon(z) \dz + \int_{\mathbb{R}} z^2 {\mathcal K}_\varepsilon(z) \dz.
	\end{equation}
	Since ${\mathcal K}_\varepsilon(z)$ is even and normed, the first integral equals to $1$  and the second vanishes. The third integral is the second moment of the Laplace distribution
	\begin{equation}
		\int_{\mathbb{R}} z^2 \left( \frac{1}{2\varepsilon} e^{-|z|/\varepsilon} \right) \dz = \frac{1}{\varepsilon} \int_0^\infty z^2 e^{-z/\varepsilon} \dz = \frac{\varepsilon^3 \Gamma(3)}{\varepsilon} = 2\varepsilon^2.
	\end{equation}
	Substituting this back, the $y^2$-terms perfectly cancel, revealing an exact match with the local viscous Burgers bound:
	\begin{equation}
		\bar I_D \leq \frac{1}{\varepsilon} \int_{\mathbb{R}} |\ow(y)| (y^2 + 2\varepsilon^2 - y^2) \,{\rm d}y = \frac{2\varepsilon^2}{\varepsilon} {\|\ow(t)\|}_{L^1(\setR)} = 2\varepsilon {\|\ow(t)\|}_{L^1(\setR)}.
	\end{equation}
	Combining these gives the differential inequality
	\begin{equation}
		 \frac{{\rm d}}{{\rm d}t}{\bar M}_2(t) \leq 2 {\|\ou\|}_{L^\infty(\setR)} \int_{\mathbb{R}} |x| |\ow| \dx + 2\varepsilon {\|\ow(t)\|}_{L^1(\setR)}.
	\end{equation}
	By the Cauchy-Schwarz inequality and the amplitude bound \eqref{radest}, we arrive at
	\begin{equation} \label{eq:M2_diff_ineq}
	\frac{{\rm d}}{{\rm d}t}{\bar M}_2(t) \leq \frac{2C}{(1+t)^{1/2}} {\|\ow(t)\|}_{L^1(\setR)}^{1/2} {\bar M}_2(t)^{1/2} + 2\varepsilon {\|\ow(t)\|}_{L^1(\setR)}.
	\end{equation}
	
	{\textbf{Step 4:} (Boundedness of auxiliary function)\\
	Again, we define ${\bar Z}(t) := (1+t)^{-1/2} {\bar M}_2(t)$. Differentiating and substituting into \eqref{eq:M2_diff_ineq} yields
	\begin{equation} \label{eq:Z_diff}
		(1+t) \frac{{\rm d}}{{\rm d}t}   {\bar Z}(t) + \frac{1}{2} {\bar Z}(t) \leq 2C {\|\ow(t)\|}_{L^1(\setR)}^{1/2} {\bar Z}(t)^{1/2} (1+t)^{1/4} + 2\varepsilon {\|\ow(t)\|}_{L^1(\setR)} (1+t)^{1/2}.
	\end{equation}
    We have recovered the same differential inequality than in \eqref{eq:Z_diffpara}. 
    As in the proof of Lemma \ref{lemma_decay}, using the same definition for $M_Z$, we 
    get 
	\[
		{\|\ow(t)\|}_{L^1(\setR)} \leq \frac{4K + 2C}{(1+t)^{1/2}} + \frac{M_Z}{K^2 (1+t)^{1/2}} = \frac{M}{(1+t)^{1/2}},
	\]
	where $M := 4K + 2C + \frac{M_Z}{K^2}$. 
    This proves the lemma. \qed
\bigskip 

\textit{Proof} of Lemma \ref{prop:Nwave_deriv}:
						Throughout the proof, $C_k>0$ denotes a generic constant that depends only on $k$, but neither on $\eps\in(0,1]$ nor on $t$.

						\textbf{Step 1:} (Notations and problem reduction)\\
						The viscous Burgers equation \eqref{RV-para} is invariant under the reflection $(x,u)\mapsto(-x,-u)$, which maps the diffusive $N$-wave with masses $(p_0,q_0)$ onto the one with masses $(q_0,p_0)$ and leaves the $L^1$-norms in the statement unchanged. Hence we may assume $p_0\le q_0$, i.e. $c_1\ge c_2$.
						Let  us define $\xi = \frac{x}{\sqrt{4\varepsilon t}}$ as Cole-Hopf variable and
						denote the derivative $\frac{d}{d\xi}$ by $'$. With this choice  there is a shape-profile function $V=V(\xi)$ such that
						\begin{equation}\label{rescaling}\Newaveo(x,t,t_0) = \sqrt{\frac{\eps}{t}} V(\xi)
						\end{equation}
						holds. Furthermore we introduce auxiliary quantities (suppressing the dependence on $t$):
						\[
						\begin{array}{c}
							\ds K = \sqrt{\frac{t}{t_0}} \frac{c_1 - c_2}{\sqrt{\pi}},\quad
							\Phi(\xi) = \frac{1}{\sqrt{\pi}} \int_{-\infty}^\xi e^{-y^2} \intd{y},\quad
							H(\xi) = c_1(1-\Phi(\xi)) + c_2 \Phi(\xi),\\[1.9ex]
							\ds E(\xi) = \sqrt{\frac{t}{t_0}} e^{\xi^2} H(\xi),\quad
							g(\xi) = \frac{1}{1+E(\xi)},\\[1.9ex]
							\ds W(\xi) = -\frac{H'(\xi)}{H(\xi)} = \frac{(c_1-c_2)e^{-\xi^2}}{\sqrt{\pi} H(\xi)},\quad
							S(\xi) = 2\xi - W(\xi).
						\end{array}
						\]
						Furthermore, we set $M_\eps := \max\{p_0,q_0\}/(2\eps) = q_0/(2\eps)$. In terms of the parameters  $A,B$ in \eqref{diffNw} the mass constants are $c_1 = 1/A$ and $c_2 = (1-B)/A$. Note that $K\ge 0$ and $W\ge 0$, that $K$ is the only quantity through which the time $t$ enters the proof apart from the trivial bound $\sqrt{t/t_0}\ge 1$, and that the hypothesis $t\le T_\eps$ from \eqref{Teps} is equivalent to
						\begin{equation}\label{Ksmall}
							K \le \eps .
						\end{equation}
						From direct computations one observes  that these quantities are related through
						\[
						V(\xi) = (2\xi + K) g(\xi), \, W(\xi) E(\xi) = K, \, E(\xi) g(\xi) = 1-g(\xi),
						\]
						thus, we have in particular
						\begin{equation}\label{Vrep}
							V(\xi) = (2\xi + K)g(\xi) =  2\xi g(\xi) + W(\xi)(1-g(\xi)).
						\end{equation}
						Moreover, using $H'' = -2\xi H'$, one checks that $W,S,E,g$ satisfy
						the ordinary differential equations
						\begin{equation}\label{ODEsystem}
							\begin{array}{rclrcl}
								W'(\xi) &\!= &\! -S(\xi)W(\xi), &
								S'(\xi) &\!= &\! 2 + S(\xi)W(\xi),\\[1ex]
								E'(\xi) &\!= &\! S(\xi) E(\xi),&
								g'(\xi) &\!= &\! -S(\xi)g(\xi)(1-g(\xi)).
							\end{array}
						\end{equation}
						Finally, $H$, being a convex combination of $c_1$ and $c_2$, exceeds the smaller of the two, $H\ge c_2$. Since moreover $t\ge t_0$, we have
						\begin{equation}\label{Elow}
							E(\xi) \ge e^{\xi^2 - M_\eps} \quad \mbox{for all } \xi\in\setR, \qquad\mbox{hence}\qquad \xi^2 \le M_\eps + |\ln E(\xi)|.
						\end{equation}

						\textbf{Step 2:} (Algebraic structure of $V^{(k)}$)\\
						From \eqref{Vrep} and \eqref{ODEsystem} we get
						\begin{equation}\label{Vprime}
							V' = 2g + A_1\, g(1-g) + B_1\, W(1-g), \qquad A_1 := -2\xi S + SW, \quad B_1 := -S .
						\end{equation}
						The terms $g(1-g)$ and $W(1-g)$ form a closed family under differentiation, namely
						\begin{eqnarray*}
							\frac{d}{d\xi}\big[g(1-g)\big] &=& g'(1-g) - g g' = -S(1-2g)\big[g(1-g)\big],\\
							\frac{d}{d\xi}\big[W(1-g)\big] &=& W'(1-g) - W g' = -S(1-g)\big[W(1-g)\big],
						\end{eqnarray*}
						the derivative of $2g$ is $-2S\big[g(1-g)\big]$, and by \eqref{ODEsystem} the derivatives of $\xi, S, W, g$ are polynomials in $(\xi,S,W,g)$. By induction, for all $k \ge 2$,
						\begin{equation}\label{AkBk}
							V^{(k)}(\xi) = A_k(\xi,S,W,g) \cdot g(1-g) + B_k(\xi,S,W,g) \cdot W(1-g),
						\end{equation}
						with $A_k$ and $B_k$ being polynomials in their variables. Each application of $\frac{d}{d\xi}$ produces a factor $S$ and thereby increases the total degree in the variables $(\xi,S,W)$ by at most one. Since $A_1$ has degree $2$, $B_1$ has degree $1$ and $(2g)'$ contributes the term $-2S$ of degree $1$, we obtain
						\begin{equation}\label{degAB}
							\deg_{(\xi,S,W)} A_k \le k+1, \qquad \deg_{(\xi,S,W)} B_k \le k .
						\end{equation}

						\textbf{Step 3:} (Pointwise bounds)\\
						For $\xi\le 0$ we have $1-\Phi(\xi) \ge \frac12$, hence $H(\xi)\ge c_1/2$ and
						\begin{equation}\label{Wleft}
							0 \le W(\xi) \le \frac{2}{\sqrt{\pi}} e^{-\xi^2} \qquad (\xi \le 0).
						\end{equation}
						For $\xi \ge 0$, substituting the definitions for $K
                        $ 
                        and $E
                        $ yields
						\begin{equation*}
							W(\xi) = \frac{c_1-c_2}{\sqrt{\pi}\, e^{\xi^2} H(\xi)}.
						\end{equation*}
						Since $c_2 \ge 0$, the numerator satisfies $c_1 - c_2 \le c_1$. Using the lower bound $H(\xi) \ge c_1(1-\Phi(\xi))$ together with the integral representation $1-\Phi(\xi) = \frac{1}{\sqrt{\pi}} \int_{\xi}^{\infty} e^{-y^2} \intd{y}$, we deduce
						\begin{equation*}
							W(\xi) \le \frac{c_1}{\sqrt{\pi}\, e^{\xi^2} c_1 (1-\Phi(\xi))} = \frac{e^{-\xi^2}}{\int_{\xi}^{\infty} e^{-y^2} \intd{y}}.
						\end{equation*}
						To control this fraction, we use an elementary Mills-ratio bound: the function $f(\xi) := 2(1+\xi)\int_\xi^\infty e^{-y^2}\intd{y} - e^{-\xi^2}$ satisfies $f'(\xi) = 2\int_\xi^\infty e^{-y^2}\intd{y} - 2e^{-\xi^2} < 0$ on $[0,\infty)$, because $\int_\xi^\infty e^{-y^2}\intd{y} = e^{-\xi^2}\int_0^\infty e^{-2\xi s - s^2}\intd{s} \le \frac{\sqrt{\pi}}{2} e^{-\xi^2}$, and $f(\xi)\to 0$ as $\xi\to\infty$. Thus $f\ge 0$, i.e.
						\[
						W(\xi) \le \frac{e^{-\xi^2}}{\int_\xi^\infty e^{-y^2}\intd{y}} \le 2(1+\xi) \qquad (\xi \ge 0).
						\]
						Together with \eqref{Wleft} we have, uniformly in $\eps$ and $t$,
						\begin{equation}\label{WSbound}
							0\le W(\xi) \le 2(1+|\xi|), \qquad |S(\xi)| \le 4(1+|\xi|), \qquad 0 < g(\xi) < 1 ,
						\end{equation}
						and consequently, by \eqref{degAB}, $|A_k| \le C_k(1+|\xi|)^{k+1}$ and $|B_k|\le C_k (1+|\xi|)^{k}$.
						Next we compare $W(1-g)$ with $g(1-g) = E g^2$. For $\xi\le 0$, \eqref{Wleft} gives $W(1-g) \le W \le \frac{2}{\sqrt\pi}e^{-\xi^2}$. For $\xi>0$ and $E\le 1$ we have $W(1-g) = W E g = W(1+E)\, g(1-g) \le 4(1+|\xi|)\, g(1-g)$, while for $\xi > 0$ and $E\ge 1$ we have $W(1-g) = Kg = \frac{K}{1+E} \le 2K \frac{E}{(1+E)^2} = 2K\, g(1-g)$. Hence, for all $\xi\in\setR$ and all $t\ge t_0$,
						\[
						W(1-g) \le \big(4(1+|\xi|) + 2K\big)\, g(1-g) + \frac{2}{\sqrt{\pi}} e^{-\xi^2} .
						\]
						Inserting this and the bounds \eqref{degAB}  on the growth of  $A_k, B_k$ into \eqref{AkBk} and \eqref{Vprime} yields the pointwise estimates
						\begin{equation}\label{pointwise}
							\begin{array}{rcll}
								|V^{(k)}(\xi)| &\!\!\le&\!\! C_k (1+|\xi|)^{k} \Big[ (1+|\xi|+K)\, g(\xi)\big(1-g(\xi)\big) + e^{-\xi^2} \Big] & (k\ge 2),\\[1.2ex]
								|V'(\xi)| &\!\!\le&\!\! 2 g(\xi) + C_1 (1+|\xi|) \Big[ (1+|\xi|+K)\, g(\xi)\big(1-g(\xi)\big) + e^{-\xi^2} \Big]. &
							\end{array}
						\end{equation}
						Under \eqref{Ksmall} this is the bound $|V^{(k)}| \le C_k(1+|\xi|)^{k+1}[g(1-g) + e^{-\xi^2}]$ for $k\ge 2$; for $K\gg \eps^{-1/2}$, i.e. in the diffusion wave regime, the factor $K$ cannot be dropped.

						\textbf{Step 4:} (The weighted integral of $g(1-g)$)\\
						For $m\ge 1$ we consider
						\begin{equation}
							\label{Vint}
							I_m
							:= \int_\setR (1+\vert{}\xi\vert{})^{m} g(\xi)(1-g(\xi))  \dxi ,
						\end{equation}
						and split the real line according to $\setR = [-1,1] \cup T_- \cup B \cup T_+$ with
						\[
						\begin{array}{c}
							T_- := (-\infty,-1), \qquad B := \{\xi\ge 1:\ E(\xi) < K/\xi\},\\[1.2ex]
							T_+ := [1,\infty)\setminus B = \{\xi\ge 1:\ E(\xi) \ge K/\xi\} .
						\end{array}
						\]
						\begin{itemize}[labelsep=0.5em, leftmargin=0.5cm,itemsep=0em, parsep=0em]
							\item[(a)] On $[-1,1]$ we have $g(1-g)\le \frac14$, so this part contributes at most $2^{m-1}$.
							\item[(b)] On $B$ we have $g(1-g) \le E < K/\xi \le K$. Moreover, by \eqref{Elow}, for $\xi \ge \xi_B := 1 + \sqrt{M_\eps + \ln^+ K}$ (with $\ln^+ K := \max\{\ln K, 0\}$) we get $E(\xi) \ge e^{\xi^2 - M_\eps} \ge e^{1+\ln^+ K} \ge K \ge K/\xi$, so that $B\subset [1,\xi_B]$ and
							\[
							\int_B (1+\xi)^m g(1-g) \dxi \le K (1+\xi_B)^{m+1} \le C_m\, K \big(1 + M_\eps + \ln^+ K\big)^{\frac{m+1}{2}} .
							\]
							\item[(c)] On $T_-$ we have $S \le 2\xi \le -2$, hence $|S| \ge 2|\xi|$, and $E' = SE < 0$, so that $\xi\mapsto E(\xi)$ is a strictly decreasing bijection of $T_-$ onto $(E(-1),\infty)$. On $T_+$ the condition $E\ge K/\xi$ means $W = K/E \le \xi$, i.e. $S \ge \xi \ge 1$. If $S(\xi')>0$ for some $\xi'$, then $S>0$ on $(\xi',\infty)$, since by \eqref{ODEsystem} $S' = 2 >0$ at any zero of $S$, so that $S$ cannot pass from positive to negative values. Hence, to the right of any point of $T_+$ the function $E$ is increasing while $K/\xi$ is decreasing, which shows that $T_+ = [\xi_1,\infty)$ is a half-line on which $\xi\mapsto E(\xi)$ is a strictly increasing bijection onto $[E(\xi_1),\infty)$ (If $p_0 = q_0$, then $K=0$, $W\equiv 0$, $S = 2\xi$, $B=\emptyset$ and $\xi_1 = 1$). On both tails we may therefore use $E$ as new variable, with
							\[
							\dxi = \frac{\intd{E}}{|S| E} \le \frac{\intd{E}}{|\xi| E}
							\]
							and, by \eqref{Elow},
							\[
							\frac{(1+|\xi|)^m}{|\xi|} \le 2^m |\xi|^{m-1} \le 2^m \big(M_\eps + |\ln E|\big)^{\frac{m-1}{2}}
							\le C_m \Big( M_\eps^{\frac{m-1}{2}} + |\ln E|^{\frac{m-1}{2}} \Big) .
							\] Since $g(1-g) = {E}/{(1+E)^2}$, we obtain
							\[
                        \begin{array}{rcl}
							\ds \int_{T_\pm} (1+|\xi|)^m g(1-g) \dxi
							&\le & \ds C_m \int_0^\infty \Big( M_\eps^{\frac{m-1}{2}} + |\ln E|^{\frac{m-1}{2}} \Big) \frac{\intd{E}}{(1+E)^2}\\[1.6ex]
							&\le& \ds C_m \big(1 + M_\eps\big)^{\frac{m-1}{2}},
                           \end{array} 
							\]
							because $\int_0^\infty (1+E)^{-2} \intd{E} = 1$ and $\int_0^\infty |\ln E|^{a} (1+E)^{-2} \intd{E} < \infty$ for each $a\ge 0$.
						\end{itemize}
						Summing up the contributions (a)--(c) gives
						\begin{equation}\label{Imbound}
							I_m \le C_m \Big[ \big(1 + M_\eps\big)^{\frac{m-1}{2}} + K \big(1 + M_\eps + \ln^+ K\big)^{\frac{m+1}{2}} \Big] .
						\end{equation}
						We also note that $\int_\setR (1+|\xi|)^{k+1} e^{-\xi^2} \dxi \le C_k$ and that, by \eqref{Elow}, $g \le \min\{1, e^{M_\eps - \xi^2}\}$ with $\xi^2 - M_\eps \ge (|\xi| - \sqrt{M_\eps})^2$ for $|\xi|\ge\sqrt{M_\eps}$, so that
						\begin{equation}\label{gint}
							\int_\setR g \dxi \le 2\sqrt{M_\eps} + 2\int_{\sqrt{M_\eps}}^\infty e^{-(\xi - \sqrt{M_\eps})^2} \dxi = 2\sqrt{M_\eps} + \sqrt{\pi} \le C \big(1+M_\eps\big)^{\frac12} .
						\end{equation}

						\textbf{Step 5:} (Conclusion)\\
						Integrating \eqref{pointwise} over $\setR$ and using $K I_k \le K I_{k+1}$, \eqref{Imbound} with $m=k+1$, and \eqref{gint} for $k=1$, we arrive at
						\begin{equation}\label{Vkcrude}
							{\Vert{}V^{(k)}\Vert{}}_{L^1(\setR)} \le C_k \big(1+K\big) \big(1+M_\eps + \ln^+ K\big)^{\frac{k+2}{2}} \qquad (k\ge 1,\ t\ge t_0).
						\end{equation}
						Under the hypothesis \eqref{Ksmall} we have $\ln^+K = 0$, $1+K\le 2$ and $\eps(1+M_\eps) = \eps + q_0/2 \le 1 + q_0/2$, so that \eqref{Imbound} improves to $I_{k+1}\le C_k (1+q_0)(1+M_\eps)^{k/2}$ and therefore
						\begin{equation}\label{Vksharp}
							{\Vert{}V^{(k)}\Vert{}}_{L^1(\setR)} \le C_k (1+q_0) \big(1+M_\eps\big)^{\frac{k}{2}} \le C_k (1+q_0)^{k+1} \varepsilon^{-\frac{k}{2}} \qquad (k\ge 1,\ 
                            t\in I_\eps).
						\end{equation}
                  We recall the definition of $I_\eps$ from \eqref{timespan}.\\      
						Finally, from \eqref{rescaling} we get $\partial_x^k \Newaveo(x,t;t_0) = \sqrt{\eps/t}\,(4\eps t)^{-k/2} V^{(k)}(\xi)$ and $\dx = \sqrt{4\eps t}\, \dxi$, hence
						\begin{equation}\label{backtransform}
							\norm{\partial_x^k \Newaveo(\cdot,t;t_0)}{L^1(\setR)} = 2^{1-k} \eps^{1-\frac{k}{2}} t^{-\frac{k}{2}} {\Vert{}V^{(k)}\Vert{}}_{L^1(\setR)} .
						\end{equation}
						Inserting \eqref{Vksharp} proves the lemma with constants $\Pi^1_k$ that depend only on $k$, $p_0$ and $q_0$.\\
						For later reference we remark that \eqref{Vkcrude} and \eqref{backtransform} give, without any restriction on $t\ge t_0$, the cruder bound
						\[
						\norm{\partial_x^k \Newaveo(\cdot,t;t_0)}{L^1(\setR)} \le C_k\, \eps^{1-\frac{k}{2}} t^{-\frac{k}{2}} \big(1+K(t)\big) \big(1+M_\eps + \ln^+ K(t)\big)^{\frac{k+2}{2}} ,
						\]
						with $K(t) = \sqrt{t/t_0}\,(c_1-c_2)/\sqrt{\pi}$, which loses a factor $\eps^{-1}(1+K(t))$ (up to logarithms) compared to the statement of the lemma, but remains valid in the diffusion wave regime $t> T_\eps$.
						\qed

 {\small

  {\bibliographystyle{abbrv}

    \bibliography{rv}}}


\printbibliography

\end{document}